%% file: MLE_Consistency_RecurrentRWRE.tex
\documentclass[11pt,a4paper]{article}

\usepackage[T1]{fontenc} 
\usepackage[utf8]{inputenc} 
\usepackage[frenchb,english]{babel} 
\usepackage{fourier}
\usepackage[scaled=0.875]{helvet}
\usepackage{courier}
\usepackage{dsfont}

\usepackage{amsmath,amsthm,amssymb,amsfonts,epic,latexsym,hyperref,natbib
}
\usepackage[dvips]{graphicx}
\usepackage[usenames]{color}
\newtheorem{theo}{Theorem}[section]

\newtheorem{defi}[theo]{Definition}
\newtheorem{prop}[theo]{Proposition}
\newtheorem{coro}[theo]{Corollary}
\newtheorem{assu}{Assumption}

\newtheorem{exam}{Example}

\newtheorem{lemm}[theo]{Lemma}
\newtheorem{rema}[theo]{Remark}

\def\R{\mathbb{R}}
\def \N{\mathbb{N}}
\def \Z{\mathbb{Z}}

\def\w{\omega}
\def\E{\mathbb{E}} 
\def\Es{\mathbb{E}^\star}

\def\P{\mathbb{P}} 
\def\Ps{\mathbb{P}^\star}

\def\PP{\mathbf{P}} 
\def\EE{\mathbf {E}} 

\newcommand{\PPs}{\mathbf{P}^{\star}}
\newcommand{\EEs}{\mathbf{E}^{\star}}

\def\1{\mathds{1}} 
\newcommand{\Var}{\mathbb{V}\mbox{ar}}

\def\pb{\mathbf p}
\def\pbs{\pb^\star}
\def\ab{\mathbf a}
\def\abs{\ab^\star}
\def\uab{{\underline{\ab}}}
\def\betab{\boldsymbol{\beta}}
\def\betas{\boldsymbol{\beta}^\star}
\newcommand{\bbeta}{\overline{ \betab}}
\def\qb{\mathbf q}
\def\t{\theta}
\newcommand{\ts}{{\theta^{\star}}}
\newcommand{\htet}{\widehat \theta_n}
\newcommand{\btet}{\overline{ \theta_n}}
\newcommand{\baa}{\overline{ \ab_n}}
\newcommand{\bpp}{\overline{ \pb_n}}

\def \eps{\varepsilon} 
\def\SS{\mathfrak{S}}
\def\FF{\mathcal{F}}

\def\to{\rightarrow}
\def\8{\infty}

\newcommand{\ii}{ {\hat \imath} }

\definecolor{lilas}{RGB}{182, 102, 210}

\newcommand{\argmax}{\mathop{\rm Argmax}}
\newcommand{\argmin}{\mathop{\rm Argmin}}
\def\dd{\mathrm{d}} 
\def\ee{\mathrm{e}} 
\def\to{\rightarrow}

\def\KL{d_{\mathrm{KL}}}
\def\nn{\nonumber}
\def\RR{\mathcal{R}}
\def\Rn{{\mathcal R}_n}
\def\GG{\mathcal{G}}
\def\DD{\mathcal{D}}
\def\VV{\mathcal{V}}

\usepackage{tikz}
\usetikzlibrary{arrows,decorations.pathmorphing,backgrounds,positioning,fit,petri} 
\usetikzlibrary{shapes}
\usetikzlibrary{decorations.shapes}
\usepackage[babel=true,kerning=true]{microtype} 

\begin{document}

\title{Maximum likelihood estimator consistency for recurrent random walk in
 a parametric random environment with finite support}

\author{F. {\sc Comets}\footnote{Laboratoire Probabilités et Modélisation Aléatoire, Universit\'e Paris Diderot, UMR~CNRS~7599, 75205 Paris cedex 13, France. 
Email: {\tt comets@math.univ-paris-diderot.fr}; 
$^\dag$ Laboratoire de Mathématiques et Modélisation d'\'Evry, Universit\'e d'\'Evry Val d'Essonne, UMR~CNRS~8071, USC~INRA, 23 Boulevard de France 91037 Evry cedex, France. 
E-mail: {\tt mikael.falconnet@genopole.cnrs.fr, dasha.loukianova@univ-evry.fr, oleg.loukianov@u-pec.fr}
}
\and
M. {\sc Falconnet}$^{\dag}$  
\and
O. {\sc Loukianov}$^{\dag}$ 
\and
D. {\sc Loukianova}$^{\dag}$ 
}

\maketitle

\begin{abstract}
\noindent We  consider  a  one-dimensional   recurrent  random  walk  in  random
environment (RWRE)  when the environment is i.i.d.  with a parametric,
finitely supported distribution. Based  on a single observation of the
path,  we provide  a maximum  likelihood estimation  procedure  of the
parameters of the environment. 
Unlike most of the classical maximum likelihood approach, the limit of
the criterion  function is in general a  nondegenerate random variable
and convergence does not hold in probability.  Not only
the leading  term but also the  second order asymptotics  is needed to
fully identify the unknown  parameter. We present different frameworks
to illustrate  these facts. We also explore  the numerical performance
of our estimation procedure.
\end{abstract}


\section{Introduction}

Since the  pioneer works  of \cite{chernov} and  \cite{temkin}, random
walks in random environments (RWRE) have attracted many probabilists and
physicists,  and the  related literature  in these  fields  has become
richer and  source of fine  probabilistic results that the  reader may
find  in surveys  including~\cite{Hughes}  and~\cite{ZeitouniSF}.  The
literature dealing  with the  statistical analysis of  RWRE is  far from being as
rich  and  we  aim at  making  a  fundamental 
contribution  to  the inference  of  parameters  of the  environment
distribution  for a  one-dimensional  nearest
neighbour path.

Let  $\w=(\w_x)_{x\in\Z}$ be an  independent and
identically distributed  (i.i.d.) collection of  $(0,1)$-valued random
variables with a parametric distribution~$\eta_\t$ of the form
\begin{equation} \label{eq:model} 
  \eta_\theta = \sum_{i=1}^d p_i \delta_{a_i},
\end{equation}
with $d$ an integer, $\pb=(p_i)_{1\leq i \leq d}$ a probability vector
and  $\ab=(a_i)_{1\leq i  \leq  d}$ the  ordered  support. We  further
assume that $d \geq 2$ is known, and the unknown 
parameter is $\t=(\ab,\pb)$. 

Denote by
$\P^{\theta}=\eta_{\theta}^{\otimes \Z}$ the law on $ (0,1)^{ \Z}$ of
the environment  $\w$ and by $\E^{\theta}$ the  expectation under this
law.  The  process~$\w$ represents a  random environment in  which the
random walk evolves. 
For  fixed environment  $\w$, let  $X=(X_t)_{t\in\Z_+}$ be  the Markov
chain on $\Z_+$ starting at $X_0=0$ and with transition probabilities 
$P_{\w}(X_{t+1}=1|X_t=0)=1$, and for $x>0$
\[  P_{\w}(X_{t+1}=y|X_t=x)=\left  \{\begin{array}{lr} \w_x&\mbox{if}\
y=x+1,\\ 1-\w_x&\mbox{if}\ y=x-1,\\ 0&\mbox{otherwise}.
\end{array}\right.
\]
For  simplicity, we  stick to  the RWRE  on the  positive integers
reflected at $0$, but our results  apply for the RWRE on the integer
axis as  well.   The symbol $P_{\w}$  denotes the measure  on the
path space of 
$X$ given 
$\w$, usually  called \emph{quenched} law. The  (unconditional) law of
$X$ is given by
\[ 
\PP^{\theta}(\cdot)=\int P_{\omega}(\cdot)\dd\P^{\theta}(\omega),
\]
this  is the  so-called \emph{annealed}  law.  We  write  $E_{\w}$ and
$\EE^{\theta}$   for   the   corresponding   quenched   and   annealed
expectations, respectively.  

Random environment is a classical paradigm for inhomogeneous media which possess 
some  regularity at  large scale.  Introduced by  \cite{chernov}  as a
model for DNA replication, RWRE was recently used by 
\cite{Balda_06, Balda_07} and \cite{Andreo}
to analyse experiments on DNA unzipping,
pointing the need of sound statistical procedures for these models.
In \eqref{eq:model} we restrict the model to environments with finite
support, a framework which already covers  many interesting applications 
and also reveals the main features of  the estimation
problem.  Considering  a
general  setup would  increase  the technical  complexity without  any
further appeal. 

The behaviour of the process $X$ is related to the ratio sequence
\begin{equation}
  \label{eq:rho} \rho_x=\frac{1-\w_x}{\w_x}, \quad x \in {\Z_+},
\end{equation}
and  we refer  to~\cite{Sol}  for the  classification  of $X$  between
transient  or  recurrent  cases  according  to  whether  $\E^{\t}(\log
\rho_0)$  is different or  not from  zero. The  transient case  may be
further  split   into  two  sub-cases,   called  \emph{ballistic}  and
\emph{sub-ballistic}  that correspond  to  a linear  and a  sub-linear
displacement for the walk, respectively.

\cite{Comets_etal}  provided a maximum  likelihood estimator  (MLE) of
the  parameter  $\t$  in  the transient  \emph{ballistic}  case.   The
estimator  maximizes  the annealed  log-likelihood  function, it  only
depends on the  sequence of the number of left  steps performed by the
random walk.   In the \emph{ballistic} transient  case this normalized
criterion  function  converges   in  probability  to  a  \emph{finite}
deterministic limit  function, which identifies the true  value of the
parameter.   \cite{Comets_etal}  establishes  the consistency  of  MLE
while  asymptotic   normality  together  with   asymptotic  efficiency
(namely,  that it  asymptotically  achieves the  Cramér-Rao bound)  is
investigated in~\cite{FLM}.  \citeauthor{DashaArnaud} proved that a slight
modification  of   the  above  criterion  function   also  provides  a
well-designed limiting  function in the  \emph{sub-ballistic} case. In
all these  works, the results rely  on the branching  structure of the
sequence of the number of left  steps performed by the walk, which was
originally observed by \cite{KKS}. 

The salient  probabilistic feature of  a recurrent RWRE is  the strong
localization  revealed by \cite{sinai}.  Inhomogeneities in  the medium
create  deep  traps the  walker  falls  into.   Typically, on  a  time
interval $[0,n]$, the  walker sneaks most of the  time around the very
bottom of  the main  ``valley'' which is  at random distance  of order
$\log^2 n$. Visiting  repeatedly the medium at the  bottom of the main
"valley",  it collects  precise information  there.  Unfortunately the
medium  at  the   bottom  is  {\em  not  typical}   from  the  unknown
distribution $\eta_\theta$, but on the contrary, it is strongly biased
by the  implicit information  that there is  a deep trap  right there.
However, the localization mechanism can be completely analysed using  the  analogy between  nearest neighbour walks and
electrical networks, as explained 
in the book  of \cite{DoSn}.  
From \cite{GPS}, we know that  the empirical distributions of the RWRE
converge weakly to a 
certain  limit law which  describes the  stationary distribution  of a
random walk in an  \emph{infinite} valley whose construction goes back
to \cite{Golosov}. 
%
This  allows to unbias  our observation  of the  medium, and  to prove
consistency of MLE. 
We  recall   at  this  point   the  moment  estimator   introduced  in
\cite{AdEn},  which  relies on  the  observation  that  the step  when
leaving a site $x$ for the  first time yields an unbiased estimator of
the environment.

In the course  of the proof, expanding the  log-likelihood for a large
observation time $n$, we prove convergence of the first order 
term  at scale  $n$ and  of  the second  order term  at scale  $\log^2
n$. Though the first order term is random in general, 
it allows to identify the support $\ab$ of 
the  environment, while  the second  one  with a  correct estimate  of
$\ab$, allows to identify the probability vector $\pb$.  
This discrepancy  reflects that the amount of  information gathered on
the unknown parameter $\ab$ is proportional 
to the duration of the observation,  whereas the one for $\pb$ is only
proportional to the number of distinct visited sites.  
As  emphasized   above,  the   expansion  of  the   likelihood,  Lemma
\ref{lemm:expansion} below, 
is  the  key.   Therefore  it  is  natural  to   introduce  a  maximum
pseudo-likelihood  estimator (MPLE),  defined by  the above  first two
terms of the expansion, 
as a intermediate step to study MLE.

The  rest   of  the   article  is  organized   as  follows.   In
Section~\ref{sect:esti}, we present the 
construction  of  our  M-estimators   (\emph{i.e.}  an  estimator
maximising some criterion function), state the assumptions required on
the       model      as       well       as      our       consistency
results.   Sections~\ref{sect:support}  and  Section~\ref{sect:probas}
are respectively devoted to the proofs of the 
consistency   result  for   the   estimators  of   $\ab$,  and   $\pb$
respectively.   Section~\ref{sect:ex}   presents   some  examples   of
environment distributions for which  we provide an explicit expression
of the limit of the criterion function, and check whether it is a nondegenerate or a constant random variable. Finally, numerical experiments 
are presented in Section~\ref{sect:simus}, focusing on the three 
examples that were developed in Section~\ref{sect:ex}.


\section{Statistical problem, M-estimators and results} 
\label{sect:esti}
We  address the  following  statistical problem:  we  assume that  the
process   $X$   is   generated   under  the   true   parameter   value
$\ts=(\ab^\star, \pb^\star)$, an interior  point of the parameter space
$\Theta$,  and we want  to estimate  the unknown  $\ts$ from  a single
observation  $(X_t)_{0 \leq  t \leq  n}$ of  the RWRE  path  with time
length $n$.  Let $d \geq 2$ an 
integer. We always assume 
that $\Theta \subset (0,1)^{2d}$ is compact  and satisfies 
Assumptions~\ref{as:recvar} and~\ref{as:ell}  below, which ensure that
the environment is recurrent and has $d$ atoms.


\begin{assu}[Recurrent environment] \label{as:recvar} For
any $\theta=(\ab, \pb)$ in $\Theta$,
\begin{equation}  \label{eq:assumpp} p_i  >0, \mbox{  for any  $i \leq
d$,} \quad \mbox{and} \quad \sum_{i=1}^d p_i \log \frac{1-a_i}{a_i}=0.
\end{equation}
\end{assu}

\begin{assu}[Identifiability] \label{as:ell} 
For  all  $\theta$  in  $\Theta$, 
\begin{equation} \label{eq:assumpa} 
0<a_1<a_2<\ldots<a_d < 1.
\end{equation}
\end{assu} 
Note  that,  since  $\Theta$  is  compact, there  exists  $\eps_0$  in
$(0,1/2d)$ such that 
\begin{equation}
  \label{equa:eps0}
  a_1 \geq \eps_0, \quad 1-a_d \geq \eps_0, \quad a_{i+1} - a_{i} \geq
  \eps_0, \quad  p_i \geq \eps_0, \quad \mbox{for any   $i \geq 1$.} 
\end{equation}


We shorten to $\PPs$ and  $\EEs$ (resp.  $\Ps$ and $\Es$) the annealed
probability $\PP^{\ts}$  and its corresponding expectation~$\EE^{\ts}$
(resp. the law of the environment~$\P^{\ts}$ and its corresponding 
expectation~$\E^{\ts}$) under parameter value~$\ts$.



Define for all $x \in \Z$,
 \begin{align}            
\label{equa:LocalTime}
&\xi(n,x):=
\sum_{t=0}^{n}\1\{X_t=x\},\\
 \label{equa:LeftStep}  
&\xi^-(n,x):=\sum_{t=0}^{n-1}\1\{X_t=x;\
X_{t+1}=x-1\},                         \\
\label{equa:RightStep}
&\xi^+(n,x):=\sum_{t=0}^{n-1}\1\{X_t=x;\
X_{t+1}=x+1\}, 
\end{align}
which are respectively  the local time of the RWRE  in $x$ and the number
of left steps (resp. right steps) from site $x$ at time $n$. Note that
\[ 
  \xi(n-1,x)=\xi^+(n,x) + \xi^-(n,x) \quad \mbox{and} \quad
|\xi^-(n,x+1)-\xi^+(n,x)|\leq 1.
\]
Denote by $R_n$ the range of the walk, 
\begin{equation}   \label{defi:range}   
R_n=  \big  \vert \RR_n\big\vert  ,  \qquad  \RR_n  = \big\{x>0  \,:\,
\xi(n-1,x) \geq 1 \big\},
\end{equation}
with $|E|$  the cardinality of  the set $E$. 
It is straightforward to  compute the quenched and annealed likelihood
of a  nearest neighbour path  $X_{[0,n]}$ of length $n$  starting from
$0$
\[ 
     P_{\w}(X_{[0,n]})=\prod_{x \in \RR_n } \w^{\xi^+(n,x)}_x
(1-\w_x)_{{\phantom{x}}}^{\xi^-(n,x)},
\]
and
\[ \PP^{\theta}(X_{[0,n]}) = \prod_{x  \in \RR_n } \int_0^1
a^{\xi^+(n,x) }(1-a)^{\xi^-(n,x)}\dd \eta_{\theta}(a).
\]
Then, the annealed log-likelihood function $\t \mapsto\ell_n(\t)$ is 
 defined for all $\theta=(\ab, \pb) \in \Theta$ as
\begin{equation} \label{eq:l} 
  \ell_n(\theta) = \log
  \PP^{\theta}(X_{[0,n]})=\sum_{x  \in \RR_n } \log \big[\sum_{i=1}^d
  a_i^{\xi^+(n,x)}(1-a_i)_{\phantom{i}}^{\xi^-(n,x)}p_i\big].
\end{equation}
\begin{defi} \label{defi:MLE}
A Maximum Likelihood Estimator (MLE) $ \htet$ of $ \ts $
is defined as a measurable choice
\begin{equation}\label{eq:estimator} 
\htet \in \argmax_{\theta\in\Theta}\ell_n(\theta).
\end{equation}
We denote $\widehat \ab_n$ and $\widehat \pb_n$ the first and second
projection of $\htet$. 
\end{defi}
Due to an analysis of the log-likelihood function provided
by    Lemma~\ref{lemm:expansion}    below,     we    are    lead    to
Definition~\ref{defi:MPLE}  of  a   Maximum  Pseudo-Likelihood
Estimator (MPLE).  To do so, some additional notations are required. 
Let
\[ 
  \nu(n,x)=     \frac{\xi(n-1,x)}{n},     \quad     \nu^+(n,x)     =
  \frac{\xi^+(n,x)}{n},   \quad  \mbox{and}  \quad   \nu^-(n,x)  =
  \frac{\xi^-(n,x)}{n}. 
\]
Let  $\Theta_\ab  =  \big\{\ab  \,   :  \,  \exists  \pb,  \,  (\ab, \pb)
\in\Theta\big\}$ be the first projection of the parameter space
$\Theta$.  Introduce for $\ab$ in $\Theta_\ab$  and $\pi^+$, $\pi^-:
\Z 
\to \R_+$ with $\sum_x [\pi^+(x)+\pi^-(x)]=1$
\begin{equation} \label{eq:L} 
  L(\ab,   \pi^+,\pi^-)  =   \sum_{x   \in  \Z}   \max_i
\big\{ {\pi^+(x)} \log a_i + {\pi^-(x)}\log (1-a_i) \big\},  
\end{equation}
and 
\begin{equation} \label{eq:Ln} 
  L_n(\ab) = L\big( \ab,\nu^+(n,\cdot),\nu^-(n,\cdot) \big),  
\end{equation}
where   we   set   $\nu^+(n,x)=\nu^-(n,x)=0$  for   any   non-positive
integer~$x$. 
Define  the increasing  sequence $\betab=(\beta_i)_{0\leq  i  \leq d}$
depending on $\ab$ as 
\begin{equation}  \label{defi:gamma} 
\beta_0=-\infty,  \quad 
\beta_i  =  \log  \Big(\frac{1-a_{i}}  {1-a_{i+1}}  \Big)  \Big/  \log
\Big(\frac{a_{i+1}} 
{a_{i}} \Big),  \mbox{ for any  $i \leq d-1$,} \quad  \mbox{and} \quad
\beta_d=\infty.
\end{equation}
For any $ x$ in $\RR_n$, define the random integer $\ii(\ab,n,x)$ as 
\begin{equation} 
  \label{defi:integerI} 
  {\ii(\ab,n,x)} = i, \qquad \mbox{if \quad $\xi^+(n,x) / \xi^-(n,x) \in
    (\beta_{i-1},\beta_{i}]$}, 
\end{equation}
which is designed to satisfy
\begin{equation}\label{defi:setI} \ii(\ab,n,x) \in 
\argmax \Big\{\xi^+(n,x)  \log a_i + \xi^-(n,x) \log  (1-a_i) \,:\, {i
\leq d} \Big\},
\end{equation}
when $x$ is in $\RR_n$. Finally, introduce $K_n(\t)$ as 
\begin{equation}
  \label{equa:Kn}
  K_n(\t)=  \frac{1} {R_n}  \sum_{x \in  \RR_n}  \log p_{\ii(\ab,n,x)}
   =   \sum_{i=1}^d \frac{R_n(\ab,i)} {R_n} \log p_i, 
\end{equation}
where $R_n(\ab,i)$ is the random integer defined by
\begin{equation}
\label{equa:Rn_i}   
  R_n(\ab,i)= \sum_{x\in\Rn } \1 \big\{ \ii(\ab,n,x)=i \big\} =
\sum_{x\in\Rn  } \1  \Big\{ \beta_{i-1}  <  \frac{\xi^+(n,x)}{\xi^-(n,x) }
\leq \beta_{i} \Big\}. 
\end{equation}
\begin{lemm} \label{lemm:expansion} 
We have
\begin{equation}
  \label{equa:expansion}
  \ell_n(\theta) = 
  n \cdot L_n(\ab) 
  + R_n \cdot K_n(\t) 
  +r_n(\t),
\end{equation}
where
\begin{equation}
  \label{equa:reste}
  r_n(\t) = 
  \sum_{x \in \RR_n} \log \left( 1 + 
    \sum_{i    \neq   \ii(\ab,n,x)    }   \frac{p_i}{p_{\ii(\ab,n,x)}}
    U_i(\ab,n,x)^{\xi(n-1,x)}   \right),
\end{equation}
with
\begin{equation}
  \label{equa:reste2}
  U_i(\ab,n,x)=\Big( \frac{a_i}{a_{\ii(\ab,n,x)}} \Big)^{\xi^+(n,x)/\xi(n-1,x)}
    \Big( \frac{1-a_i}{1-a_{\ii(\ab,n,x)}} \Big)^{\xi^-(n,x)/\xi(n-1,x)}.
\end{equation}
Furthermore, for any $\t \in \Theta$,
\begin{equation}
  \label{equa:reste3}
  \frac 1  n \Big(R_n  \cdot K_n(\t) +  r_n(\t) \Big)  \xrightarrow[ n
  \to \infty]{} 0, \quad \mbox{$\PPs$-a.s.,}
\end{equation}
as well as
\begin{equation}
  \label{equa:reste4}
  \frac{r_n(\t)}{\log^2 n} \xrightarrow[ n
  \to \infty]{} 0,\quad  \mbox{in $\PPs$-probability}.
\end{equation}
\end{lemm}


Hence, the first term in the RHS of \eqref{equa:expansion} is of order
$n$ and depends only on $\ab$, whereas the second term is of order 
$\log^2 n$ and depends on $\ab$ and $\pb$.  The proofs 
of~\eqref{equa:reste3} and~\eqref{equa:reste4} are respectively 
provided  in  Sections~\ref{sect:L} and~\ref{sect:asymptotique2}.   As
claimed  above, in view  of Lemma~\ref{lemm:expansion},  the following
definition appears natural for an estimator of $\ts$. 

\begin{defi} \label{defi:MPLE}
  A   Maximum   Pseudo-Likelihood  Estimator   (MPLE)   $
\btet=(\baa,\bpp)$    of    $\theta$    is   a    measurable    choice
of
\begin{equation}\label{eq:estimatorPMLE} 
\left\{
\begin{array}{ccl}  \baa  &   \in  &  \displaystyle  \argmax_{\ab  \in
    \Theta_{\ab}} \ L_n(\ab), \\
\bpp  & \in &  \displaystyle \argmax_{\pb} \ 
K_n(\baa,\pb).
\end{array} \right.
\end{equation}
\end{defi}
In the beginning of Section~\ref{sect:probas}, we prove that 
\begin{equation} \label{defi:tpb} 
\bpp= \left(\frac{R_n(\baa,i)}{R_n} \,:\,
i=1,\ldots,d\right).
\end{equation}
Now, we can state our consistency results. 




\begin{theo} \label{theo:cvsupport} 
 Let Assumptions~\ref{as:recvar} and~\ref{as:ell} hold. 
Both the  ML estimator $\widehat  \ab_n$ and the MPL  estimator $\baa$
converge in $\PPs$-probability to the true parameter value~$\abs$. 
\end{theo}

\begin{theo} \label{theo:cvprobas} 
 Let   Assumptions~\ref{as:recvar}
and~\ref{as:ell} hold. Both the  ML estimator $\widehat \pb_n$ and the
MPL  estimator  $\bpp$ converges  in  $\PPs$-probability  to the  true
parameter value~$\pbs$.  
\end{theo}



We expect the speed of convergence for estimating $\pb$ is much slower
than the one for estimating $\ab$. This is supported by our simulation
experiment  provided  in Section~\ref{sect:simus}  but  we leave  this
question for further research.  Section~\ref{sect:support} is devoted to
the proof of the consistency of $\baa$ and $\widehat \ab_n$ whereas 
Section~\ref{sect:probas} is  devoted to the proof  of the consistency
of $\widehat \pb_n$ and $\bpp$.

\textbf{Concluding remarks.} 
Let us start  to describe a naive estimation  based on recurrence. For
all $x$ in $\RR_n$, we can estimate the environment at this point by 
\[
\hat \omega_x^{(n)} = \frac{\xi^+(n,x)}{\xi(n-1,x)}.
\]
By recurrence, $\hat \omega_x^{(n)}$ converges to $\omega_x$,
$\PPs$-a.s. With some extra work, it can be shown that, 
\begin{equation*}
  \frac{1}{R_n}  \sum_{x  \in  \RR_n}  \delta_{\hat  \omega_x^{(n)}  }
  \xrightarrow[n \to \infty]{} \eta_\ts 
  \quad \mbox{$\PPs$-a.s.}, 
\end{equation*}
 where $\RR_{n}$ is defined by~\eqref{defi:range}, and this leads 
to estimators of the parameters.  This empirical estimator is then 
consistent. However, it gives equal weight to all visited sites $x$ in
$\Rn$,  without any notice  of the  number of  visits there,  which is
certainly far from optimal. This is essentially how \cite{AdEn} devise
their  estimators,  and the  simulations  in Section  \ref{sect:simus}
indicate  they perform  poorly  at  some extend.  On  the other  hand,
\cite{andreo_sinai}  proposes estimators  based on  local time  of the
walk, which  indeed take  care of  this flaw, but  which are  adhoc in
essence and difficult  to use in an optimal  manner.  In contrast, our
estimate  relies on  first principles - maximum likelihood -  and uses  the  full information
gathered by the walk all through.  


\section{Proof  of consistency  for the MLE and MPLE of  the ordered
  support} \label{sect:support}

In the  present section, we  first recall the weak  convergence result
established by~\citeauthor{GPS} for the empirical distributions of the
RWRE.  Then, we  identify the limit of our criterion  as a functional of
$\ab$ only, and provide some information on it.  Using its regularity
properties, we can adapt the  proof of consistency for M-estimators to
our context.  




\subsection{Potential and infinite valley}
\label{sect:potential}

The environment $\w$ in which the random walk evolves is visualized as
a potential landscape $V$ where $V = \{ V(x) \, : \,  x \in \Z \}$ is
defined by 
\begin{equation} \label{equa:Pot} 
V(x) = \left\{
    \begin{array}{ll} \sum_{y=1}^x \log \rho_y & \mbox{if $x>0$,} \\ 0
& \mbox{if $x=0$,} \\ - \sum_{y=x+1}^0 \log \rho_y& \mbox{if $x<0$,}
    \end{array} \right.
\end{equation}
with $\rho_y = (1-\w_y)/\w_y$. An  example of a realization of $V$ can
be seen on Figure~\ref{figu:potential}.

Setting $C(x,x+1)=\exp[-V(x)]$,
for any  integer $x$, the  Markov chain $X$  is an
electric network in the sense of \cite{DoSn} or \cite{LePeWi}, where
$C(x,x+1)$ is the conductance  of the (unoriented) bond $(x,x+1)$.  In
particular, the measure $\mu$ defined as 
\begin{equation} \label{equa:Mu} 
\mu(x) = \exp[-V(x-1)]+\exp[-V(x)], \quad x \in \Z,
\end{equation}
is a reversible and invariant measure for the Markov chain $X$.

Now, define the right border $c_n$ of the ``valley'' with depth $\log n +
  (\log n)^{1/2}$ as
\begin{equation} \label{equa:c_n}
  c_n = \min \big\{x \geq 0 \,  : \, V(x) - \min_{0\leq y \leq x} V(y)
  \geq \log n + (\log n)^{1/2} \big\}, 
\end{equation}
and the bottom $b_n$ of the ``valley''  as
\[ 
  b_n = \min  \big\{x \geq 0 \, :  \, V(x) = \min_{0 \leq  y \leq c_n}
  V(y)\big\}. 
\]
On Figure~\ref{figu:potential}, one can see a representation of $b_n$ and
$c_n$. We are  interested in the shape of  the ``valley'' $(0,b_n,c_n)$
when  $n$ tends  to infinity  and we  recall the  concept  of infinite
valley introduced by \cite{Golosov}.

\begin{figure}
\begin{center}
\input{FigurePotentiel.tex}
\end{center}
\caption{Example  of  potential  derived  from  a  random  environment
 distributed  as in Example~\ref{ex:Temkin}
  with parameter $a=0.3$. Simulation with $n=1000$.}
\label{figu:potential}
\end{figure}
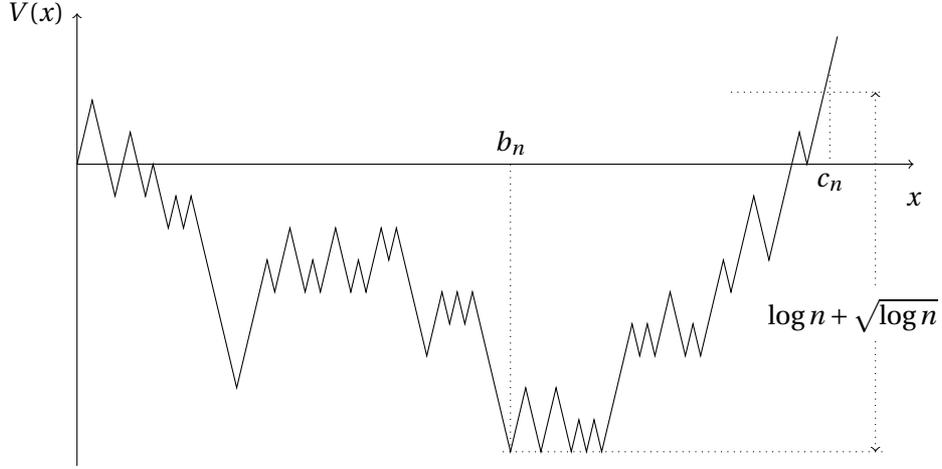

Let $\widetilde  V =  \{ \widetilde V(x)  \, :  \, x \in  \Z \}$  be a
collection of random variables  distributed as $V$ conditioned to stay
positive   for   any   negative   $x$,   and   non-negative   for   any
non negative~$x$.  For simplicity,  we assume that without loss of
generality, $\widetilde V$ is also realised under $\Ps$.  
For each realization of $\widetilde V$, consider the
corresponding  Markov chain on  $\Z$, which  is an  electrical network
with conductances $\widetilde  C(x,x+1) = \exp[-\widetilde V(x)]$.  As
usual, the measure  $\widetilde \mu$ defined as $  \widetilde \mu(x) =
\widetilde  C(x-1,x)  +  \widetilde  C(x,x+1)$,  $ x  \in  \Z$,  is  a
(reversible)  and   invariant  measure  for  the   Markov  chain  $X$.
Furthermore, the measure  $\widetilde \mu$ can be normalized  to get a
reversible probability measure $\nu$, defined by
\begin{equation} \label{equa:Nu} 
 \nu(x) = \frac{ \exp[-\widetilde 
   V(x-1)]+\exp[-\widetilde  V(x)] }{2\sum_{z\in  \Z} \exp[-\widetilde
   V(z)]} . 
\end{equation}
Note  that  $\nu(x)$ can  be  written as  the  sum  of $\nu^+(x)$  and
$\nu^-(x)$, where for any $x \in \Z$
\begin{equation} \label{equa:Nu2} 
  \nu^+(x)  =   \frac{  \exp[-  \widetilde   V(x)]  }{2\sum_{z\in  \Z}
    \exp[-\widetilde  V(z)]  }   \quad  \mbox{and}  \quad  \nu^-(x)  =
  \frac{ \exp[-\widetilde V(x-1)]  }{2\sum_{z\in \Z} \exp[- \widetilde
    V(z)] } . 
\end{equation} 
We have $\nu^+(x) =\nu^-(x+1) $.
Define $\widetilde \omega (x) \in (0,1)$, for any $x \in \Z$, as
\begin{equation} \label{equa:tildeOmega} 
  \widetilde \omega (x) = \frac {\nu^+(x)} {\nu(x)} = \frac 1 {1 +
    \exp[\widetilde V(x) - \widetilde V(x-1)]}.
\end{equation}

\begin{rema} \label{rema:valeurs_omega_tilde}
Noting that the  possible values of $\widetilde V(x)  - \tilde V(x-1)$
are those of $V(x) - V(x-1)$ for any integer $x$, we deduce that under
$\Ps$,  $\widetilde  \omega(x)$ is equal to one  of the coordinates
of $\abs$. 
\end{rema}

\cite{GPS}  showed  that  the  empirical distribution  of  the  RWRE,
suitably centered at $b_n$ converges to the stationary distribution of
a random walk in an infinite valley. More precisely, let
\[ 
  \nu_n(x)=  \nu(n,x+b_n), \quad  \nu^+_n(x)  = \nu^+(n,x+b_n),  \quad
  \mbox{and} \quad \nu^-_n(x) = \nu^-(n,x+b_n). 
\]

\begin{theo}[\cite{GPS}] 
Let    Assumptions~\ref{as:recvar}    and~\ref{as:ell}   hold.     The
distributions of $\{\nu_n(x)  \, : \, x \in \Z  \}$ converge weakly to
the  distribution  of  $\nu$  (as  probability  measures  on  $\ell^1$
equipped with the $\ell_1$-norm). 
As a consequence, for each  strongly continuous functional $f : \ell^1
\to \R$  which is shift  invariant, we have
\[ 
f \big( \{  \nu(n,x) \, : \, x \in \Z \} \big) \xrightarrow[n \to
\infty]{\mbox{{\small  law}}} f \big(  \{\nu(x) \,  : \,  x \in  \Z \}
\big).
\]
\end{theo}
In  \cite{GPS} is mentioned that the result still  holds with obvious extensions for the
non-reflected case, i.e., of a RWRE on $\Z$. An inspection of the proof of \citeauthor{GPS} immediatly yields:

\begin{coro} \label{coro:GPS} 
Let Assumptions~\ref{as:recvar} and~\ref{as:ell} hold.  The
distributions of 
\[
\{(\nu^+_n(x),\nu^-_n(x)) \, : \, x \in \Z \}
\]
converge weakly to the  distribution of $\{(\nu^+(x),\nu^-(x)) \, : \,
x \in \Z \}$. 
As a consequence, for each  strongly continuous functional $f : \ell^1
\times \ell^1 \to \R$ which is shift invariant, we have
\[ 
  f \big( \{ (\nu^+(n,x),\nu^-(n,x))  \, : \, x \in \Z \} \big)
  \xrightarrow[n \to \infty]{\mbox{{\small law}}} f
  \big( \{(\nu^+(x),\nu^-(x)) \, : \, x \in \Z \} \big).
\]
\end{coro}


\subsection{Identification and properties of the criterion limit} 
\label{sect:L}

First, we start with 


\begin{proof}[Proof of \eqref{equa:reste3} in Lemma~\ref{lemm:expansion}]
  Clearly,
  \begin{equation}
    \label{equa:mino_elln}
    \log \eps_0  \times \max_{0 \leq t \leq n}
    X_t \leq {\ell_n(\t)} - \sum_{x \in \RR_n} \max_{i \leq d}
    \Big\{\xi^+(n,x) \log a_i + \xi^-(n,x) \log (1-a_i) \Big\}\leq 0,
  \end{equation}
  with $\eps_0$ from~\eqref{equa:eps0}, and from~\eqref{eq:L} and~\eqref{eq:Ln}
  \begin{equation*}
    \sum_{x \in \RR_n} \max_{i \leq d} \Big\{\xi^+(n,x) \log a_i + \xi^-(n,x)
    \log         (1-a_i)        \Big\}        =         n        \cdot
    L_n(\ab).
  \end{equation*}
  Since $\EEs(\rho_0)  \geq \exp[ \EEs \log(\rho_0)] =1$,  case (c') in
  Theorem 2.1.9 in \cite{ZeitouniSF} applies,  and $X_n / n $ converges to
  0, $\PPs$-a.s. Hence, $\frac 1 {n} \max_{0 \leq t \leq n} X_t $ converges to
  0 $\PPs$-a.s and the claim is proved.
\end{proof}

For any $\ab \in \Theta_\ab$, denote $L_\8(\ab)$ 
\begin{equation}
  \label{eq:Linfini}
  L_\8(\ab) = L(\ab,\nu^+,\nu^-) = \sum_{x \in \Z} \max_i 
\big\{ {\nu^+(x)} \log a_i + {\nu^-(x)}\log (1-a_i) \big\},
\end{equation}
where $\nu^+$ and $\nu^-$ are defined
by~\eqref{equa:Nu2}. Anticipating~Lemma~\ref{lemm:Llip} below,
Corollary~\ref{coro:GPS} and~\eqref{equa:reste3} immediatly yields 
\[ 
  L_n(\ab) \xrightarrow[n \to \infty]{\mbox{{\small law}}} L_\8(\ab). 
\] 



Therefore,    we   provide    some   useful    information    on   the
functional~$L_\8(\cdot)$.  To  do so,  some  additional notations  are
required. 

\begin{defi}[Boltzmann entropy function and Kullback-Leibler
  distance] \label{def:ekl} 
Define the Boltzmann entropy function $H(\cdot)$ on $(0,1)$ as
\[ 
  H(q) = - [q \log q +(1-q) \log(1-q)] \geq 0,
\]
the Kullback-Leibler  distance $\KL(\cdot  | \cdot)$ on  $(0,1) \times
(0,1)$ as
\[ \KL(q|q')  = q \log  \frac q {q'}  + (1-q) \log \frac  {1-q} {1-q'}
\geq 0,
\]
and  their multinomial  extensions  on probability  vectors $\qb$  and
$\qb'$
\[ H(\qb) = - \sum_i q_i  \log q_i\geq 0, \quad \KL(\qb|\qb') = \sum_i
q_i \log \frac {q_i} {q'_i} \geq 0.
\]
\end{defi}

With $\nu$ and $\widetilde \omega$ defined by \eqref{equa:Nu} and
\eqref{equa:tildeOmega}, we have for all $\ab$ in $\Theta_a$ the 
identity 
\begin{equation} 
  \label{eq:Lentropy} 
  L_\8(\ab)  = - \sum_{x \in
    \Z} \nu(x)  H\big[\widetilde \omega (x)\big] - \sum_{x  \in \Z} \nu(x)
  \min_i \big ( \KL\big[ \widetilde \omega(x) | a_i \big] \big).
\end{equation}
From~Remark~\ref{rema:valeurs_omega_tilde},    we    have   $\min_i    \big
(  \KL\big[ \widetilde  \omega(x)  | a_i^\star  \big] \big)=0$,  and
using the fact that $ \KL(q|q')> 0$ for $q\neq q'$, we deduce that 
\begin{equation} \label{eq:La*} 
  L_\8(\ab)   <   L_\8(\ab^\star)  =   -   \sum_{x   \in  \Z}   \nu(x)
  H\big[\widetilde \omega (x)\big], \qquad \ab \neq \ab^\star. 
\end{equation}
More precisely, a useful bound is
\begin{equation} \label{eq:L>} 
  L_\8(\ab)
  \leq 
  L_\8(\abs) 
  - 
  \min_k  \big(  \nu \big\{  x  \in\Z  \,:\, \widetilde  \omega(x)
  =a^\star_k \big\} \big) 
  \times 
  \sum_{j=1}^d \min_i \big ( \KL\big[ a_j^\star | a_i \big] \big), 
\end{equation}
where the sum is deterministic  and positive for $\ab \neq \ab^\star$,
though  its factor  ($\min_j(\cdot)$) is  a.s. positive  and  does not
depend on $\ab$.

\begin{lemm}\label{lemm:Llip} 
The function $L$ defined by \eqref{eq:L} is Lipschitz continuous:
\begin{align}
\label{eq:lemm:Llip1} |  L(\ab, \pi^+,\pi^-)-L(\ab, \mu^+,\mu^-)| &\leq
|\log \eps_0| \cdot \big( \|\pi^+-\mu^+\|_{1} +\|\pi^- - \mu^-\|_{1}\big),\\
\label{eq:lemm:Llip2}   |L(\ab,\pi^+,\pi^-)-L(\ab',\pi^+,\pi^-)|  &\leq
2\eps_0^{-1} \|\ab-\ab'\|_2.
\end{align}
\end{lemm} 
\begin{proof}[Proof of Lemma~\ref{lemm:Llip}]  
Recall  that $\R^d\ni  u \mapsto  \max_i  u_i \in  \R$ is  1-lipschitz
continuous for the norm~$\|\cdot\|_\8$.  Moreover, the mapping
\[ 
  f_i \, : \, (\ab,v^+,v^-) \mapsto v^+ \log a_i + v^- \log (1-a_i)
\]
is $|\log \eps_0|$-lipschitz continuous in $v^+$, resp. in $v^-$, with
$\eps_0$ from~\eqref{equa:eps0}.
By composition, $(v^+,v^-)  \mapsto \max_i f_i(\ab,v^+,v^-)$ is $|\log
\eps_0|$-lipschitz   continuous  in   the   $\|\cdot\|_1$  norm,   and
\eqref{eq:lemm:Llip1} follows by summing over $x$.
By Cauchy-Schwarz,
\begin{align*} 
|   f_i(\ab,v^+,v^-)-f_i(\ab',v^+,v^-)|  &=  \Big\vert
\int_0^1                                                   {\partial_a}
f_i(\ab'+t(\ab-\ab'),v^+,v^-)\cdot(\ab-\ab')\dd t \Big\vert\\
&\leq    \|\ab-\ab'\|_2    \times    \sup_{\ab"}    \|    {\partial_a}
f_i(\ab",v^+,v^-)\|_2 ,
\end{align*}
where derivative can be bounded using
\[ 
  \| {\partial_a} f_i(\ab,v^+,v^-)\|_2=
  \left( a_i^{-2}(v^+)^2+(1-a_i)^{-2}(v^-)^2\right)^{1/2} \leq
  2\eps_0^{-1} (v^++v^-).
\]
Taking  the  maximum  over  $i$  and summing  over  $x$,  this  yields
\eqref{eq:lemm:Llip2}.  
\end{proof}






\subsection{Proof of Theorem \ref{theo:cvsupport}}
Recall that $\baa$ is defined by~\eqref{eq:estimatorPMLE}. 
Fix some $\eps>0$.  We prove that for all  $\eps_1>0$, there exists an
integer $n_1$ such that for all $n \geq n_1$,
\begin{equation} \label{equaconv_proba_a}
  \PPs  \big(  \baa \in  {\mathcal  B}(\ab^\star,  \eps  ) \big)  \geq
  1-2\eps_1, 
\end{equation}
where ${\mathcal B}(\ab^\star,\eps)$ is the open ball of radius $\eps$
centered at $\ab^\star$. 

Under $\Ps$, the sequence $(g[\widetilde \omega (x)] \, : \, x \in
\Z)$ with  $g(u)= \log(u^{-1}-1)$  is the sequence  of the jumps  of a
random  walk   $V$  conditioned  to   stay  positive.   The   jump  of
unconditionned   $V$  takes  the   value  $g(a_j^\star)$   with  positive
probability $p_j^\star$. Hence  it is not difficult to  check that for
all $j$  in $\{1,\dots,d\}$, we  can find $\Ps$-a.s.  some  random~$x$
such   that  $\widetilde  \omega(x)   =a^\star_j$.   Since   $\nu$  is
$\Ps$-a.s. supported by the whole $\Z$, we have
\[
  \Ps\big(   \min_j   \big\{  \nu   \{x   \,:\,  \widetilde   \omega(x)
  =a^\star_j\} \big\} >0\big)=1. 
\]
By continuity,  for arbitrary $\eps_1>0$ we can  fix $\delta_1>0$ such
that 
\begin{equation}\label{eq:SG} 
  \Ps   \Big(   \min_j  \big\{   \nu   \{x  \,:\,\widetilde   \omega(x)
  =a^\star_j\} \big\} \geq \delta_1 \Big] \geq 1- \eps_1.
\end{equation}
By compactness of $K=\Theta_\ab \setminus {\mathcal B}(\ab^\star,\eps)$, 
\[ 
  \kappa = \inf \Big\{ \sum_{j=1}^d \min_i \big \{ \KL\big[ a_j^\star
    | a_i \big] \big\} \,:\, \ab \in K \Big\} >0. 
\]
By \eqref{eq:lemm:Llip2},
\begin{equation} \label{eq:chiant} 
  \ab' \in  {\mathcal B}\big(\ab,  \kappa \delta_1 \eps_0/6  \big) \Rightarrow
  |L_\8(\ab) - L_\8(\ab')| \leq \kappa \delta_1 /3. 
\end{equation}
Since  $K$  is  totally  bounded,  we can  select  a  finite  covering
$\cup_{k=1}^{k_0} {\mathcal  B}(\ab^{k}, \kappa \delta_1  \eps_0/6 ) $
of $K$  with balls of  that radius.  From \eqref{eq:lemm:Llip1}  we can
apply Corollary \ref{coro:GPS}, and we have
\[ 
  \big( L_n(\ab^k) \,:\, 0 \leq k \leq
  k_0\big) \xrightarrow[n \to \infty]{\mbox{{\small law}}} 
  \big(L_\8(\ab^k) \,:\, 0 \leq k \leq k_0\big), 
\]
where we have set $\ab^0=\ab^\star$. Also, 
\[ 
\big( 
L_n(\abs) - L_n(\ab^k) 
\big)_{k \leq k_0}
\xrightarrow[n \to \infty]{\mbox{{\small law}}} 
\big( L_\8(\ab^\star) - L_\8(\ab^k)\big)_{k \leq k_0},
\]
where the limits are simultaneously larger than $\kappa \delta_1$ on a
set of large probability (larger than $1-\eps_1$) by \eqref{eq:L>} and 
\eqref{eq:SG}.  Then, we find $n_1$ such that for $n \geq n_1$,
\[ 
\PPs \big( L_n(\ab^\star) - L_n(\ab^k) 
\geq 2 \kappa \delta_1 /3, \,1 \leq k \leq k_0\big) \geq 1-2\eps_1. 
\]
Taking \eqref{eq:chiant} into account, we obtain that, for all $n \geq
n_1$,
\[ 
  \PPs \big( L_n(\ab^\star) - L_n(\ab) \geq 
  \kappa \delta_1 /3, \ab \in K \big) \geq 1-2\eps_1. 
\]
This  implies \eqref{equaconv_proba_a}  for  all $n  \geq n_1$.  Hence
$\baa \to \ab^\star$ in $\PPs$-probability. Now, we turn to the
convergence of $\widehat \ab_n$ to $\ab^\star$ in
$\PPs$-probability. According to~\eqref{equa:mino_elln}, we have
\[
L_n(\ab) - u_n \leq \frac {\ell_n(\t)} n \leq L_n(\ab), 
\]
where $u_n$ is non-negative and converges $\PPs$-a.s. to $0$
independently from $\t$. 
Choosing $n_2$ such that for all $n \geq n_2$ 
\[
\PPs \big( u_n \geq \kappa \delta_1 /3 \big) \leq \eps_1,
\]
achieves the proof of the convergence $\widehat \ab_n \to \ab^\star$. 


\section{Proof  of consistency of  estimates for  the probability
  vector} \label{sect:probas} 

First, note that $K_n(\t)$ can be rewritten 
\begin{equation}
  \label{equa:KnBis}
  K_n(\t)   =   -   H   \Big(   \frac{R_n(\ab,\cdot)}{R_n}   \Big)   -
  \KL\Big(\frac{R_n(\ab,\cdot)}{R_n} \Big| \pb \Big). 
\end{equation}
Therefore, 
\begin{align*} 
K_n(\baa,\pb) 
& \leq - H \Big( \frac{R_n(\baa,\cdot)}{R_n} \Big),
\end{align*}
with  equality if  and  only if  $\pb= \left(\frac{R_n(\baa,i)}{R_n}  \,:\,
i=1,\ldots,d\right)$. Hence, we can  use the alternative definition of
$\bpp$ given by \eqref{defi:tpb} to prove Theorem~\ref{theo:cvprobas}.


\subsection{Proof of Theorem~\ref{theo:cvprobas}: convergence of $\bpp$}
For $0<\delta<1$ let
\begin{align} 
\label{equa:G_n} \GG_n^{\delta}&= \big\{ x \leq b_n \,:\, \max_{z \in
[x,b_n]} V(z) - V(x) \geq \delta \log n \big\}, \\ 
\label{equa:D_n} \DD_n^\delta &=
\big\{ x >b_n \,:\, \max_{z \in [b_n,x]} V(z) - V(b_n) \leq (1-\delta)
\log n \big\}, \\ 
\label{equa:R_n} \Rn^\delta &=  \GG_n^\delta  \cup  \DD_n^\delta.
\end{align}
Note that  $\DD_n^\delta$ is an interval, $\DD_n^\delta=]b_n, c_n^\delta]$
with $  c_n^\delta \leq c_n$.

Denote    $\betas$     and    $\bbeta$    the     sequences    defined
by~\eqref{defi:gamma}   replacing   $\ab$   by   $\abs$   and   $\baa$
respectively.  Noting that for any $i$ in 
$\{1,\dots,d-1\}$,  $\KL(a_i^\star|a_{i+1}^\star)>0$,  we  can  choose
$\eps'>0$ small enough such that for any $i \in \{1,\ldots, d \}$
\begin{equation}    \label{equa:choix_eps}    \beta_{i-1}^\star   \leq
\frac{a^\star_i}{1-a_i^\star}           -           2\eps'          <
\frac{a^\star_i}{1-a_i^\star} + 2\eps' \leq \beta_i^\star.
\end{equation}
Let $\eps>0$.   Define the events  $A_n^\delta(\eps')$, $B_n(\eps')$
and $C_n^\delta(\eps,i)$ by
\begin{align} 
\label{equa:eventA}  A_n^\delta(\eps') &= \left\{\forall x \in \Rn^\delta
\,  : \,  \Big| \frac{\xi+(n,x)}{\xi^-(n,x)}  -  \frac 1{\rho_x}
\Big| \leq \eps' \right\}, \\ 
\label{equa:eventB} B_n(\eps') &= \left\{\forall i \in
\{1,\ldots, d\}  \, : \,  | \overline{\beta_i} - \beta_i^\star  | \leq
\eps' \right\},\\ 
\label{equa:eventC} C_n^\delta(\eps,i) &= \left\{\Big| \frac 1 {|\Rn \cap
\Rn^\delta|} \sum_{x\in\Rn \cap \Rn^\delta} \1 \{ \w_x = a_i^\star\} -
p_i^\star \Big| > \frac{\eps}{4} \right\},
\end{align}
Denote $^{\complement}A_n^\delta(\eps')$ and
$^{\complement}B_n(\eps')$ the respective complementary events of
$A_n^\delta(\eps')$ and 
$B_n(\eps')$, and define the quantities $\phi_n^\delta(\eps')$,
$\psi_n^\delta(\eps)$ and $\mu_n^\delta(\eps)$ by 
\begin{align*} 
  \phi_n^\delta(\eps') &= \PPs
  \big(^{\complement}A_n^\delta(\eps')        \big)        +       \PPs
  \big(^{\complement}B_n(\eps') \big)\\
  \psi_n^\delta(\eps)     &=    \PPs     \Big(    \frac{|\Rn\setminus
    \Rn^\delta|}{R_n} > \frac{\eps}{2} \Big), \\ 
  \mu_n^\delta(\eps,  i ) &= \PPs \Big( C_n^\delta(\eps,i) \Big).
\end{align*}
Then, we have
\begin{equation} \label{equa:domi_proba}
  \PPs\Big  (  \big| \bpp(i)  -  p_i^\star  \big|  > \eps  \Big)  \leq
  \phi_n^\delta(\eps') + \psi_n^\delta(\eps)  +
  \psi_n^\delta(\eps/2)   
  + \mu_n^\delta(\eps,  i ).
\end{equation}
Indeed, it is clear that
\begin{equation} \label{equa:domi_proba1} 
  \PPs\Big ( \big| \bpp(i) - p_i^\star \big| > \eps \Big) \leq
\PPs\Big(|\bpp(i)    -   p_i^\star|   >    \eps,   A_n^\delta(\eps'),
B_n(\eps')\Big ) + \phi_n^\delta(\eps'),
\end{equation}
and  writing the  set $\Rn$  as  the union  of the  two disjoint  sets
$\left(\Rn   \cap  \Rn^\delta   \right)$   and  $\left(\Rn   \setminus
\Rn^\delta \right)$ in~\eqref{equa:Rn_i} and using~\eqref{defi:tpb} yields
\begin{equation} \label{equa:domi_proba2}  
  \PPs\Big(|\bpp(i)   -   p_i^\star|   >   \eps,   A_n^\delta(\eps'),
  B_n(\eps')\Big     )     \leq     \PPs\Big(     A_n^\delta(\eps'),
  B_n(\eps'),D_n^\delta(\eps,i) \Big) + \psi_n^\delta(\eps), 
\end{equation}
with
\[ 
D_n^\delta(\eps,i)= \left\{\Big| \frac  1 {R_n}  \sum_{x\in\Rn \cap
\Rn^\delta}        \1        \Big\{\overline        {\beta_i}        <
\frac{\xi^+(n,x)}{\xi(n,x)  }  \leq  \overline{\beta_{i+1}}  \Big\}  -
p_i^\star \Big| > \frac{\eps}{2} \right\}.
\]
Using  our choice  of $\eps'$  to satisfy  \eqref{equa:choix_eps}, we
have for all $i$ in $\{1,\dots,d\}$
\[
  \left\{\overline  {\beta_i}  <  \frac{\xi^+(n,x)}{\xi(n,x)  }  \leq
\overline{\beta_{i+1}},A_n^\delta(\eps'),B_n(\eps')     \right\}    =
\left\{\omega_x=a_i^\star,A_n^\delta(\eps'),B_n(\eps') \right\}.
\]
Hence,
\begin{equation} \label{equa:domi_proba3} 
  \PPs\Big(  A_n^\delta(\eps'),  B_n(\eps'),D_n^\delta(\eps,i) \Big)
  \leq \PPs  \Big( \Big| \frac 1 {R_n}  \sum_{x\in\Rn \cap \Rn^\delta}
  \1 \{ \w_x = a_i^\star\} - p_i^\star \Big| > \frac{\eps}{2} \Big),
\end{equation}
and clearly,
\begin{equation} \label{equa:domi_proba4} 
  \PPs  \Big( \Big| \frac  1 {R_n}  \sum_{x\in\Rn \cap  \Rn^\delta} \1
  \{ \w_x = a_i^\star\} - p_i^\star \Big| > \frac{\eps}{2} \Big)  
  \leq  \psi_n^\delta(\eps/2)  + \mu_n^\delta(\eps,  i ).
\end{equation}
Combining      \eqref{equa:domi_proba1},     \eqref{equa:domi_proba2},
\eqref{equa:domi_proba3}     and    \eqref{equa:domi_proba4}    yields
\eqref{equa:domi_proba}. 

Anticipating some  lemmas which are  proved independently in  the next
section, we conclude the proof. 
From Lemma~\ref{lemm:Z}, we have
\begin{equation*}
  \lim_{\delta \to 0} \ \limsup_{n \to \8} \ \psi_n^\delta(\eps) 
  = 0.
\end{equation*}
By the law of large numbers, 
\begin{equation*}
  \lim_{n \to \8} \mu_n^\delta(\eps) = 0.
\end{equation*}
In  view   of  \eqref{equa:domi_proba},  to  conclude   the  proof  of
Theorem~\ref{theo:cvprobas} it suffices 
to prove that, for  all $\delta>0$, $\phi_n^\delta(\eps) $ vanishes as
$n \to \8$.  
On   the    one   hand,    $\baa$   converges   to    $\ab^\star$   in
$\PPs$-probability, and thus 
$$ 
 \PPs \big(^{\complement}B_n(\eps') \big) \longrightarrow 0.$$
On the other hand, by Lemma~\ref{lemm:xi_omega}, 
 $$
 \PPs
\big(^{\complement}A_n^\delta(\eps')        \big)   \longrightarrow 0.  $$
This concludes  the proof of the convergence  in $\PPs$-probability of
$\bpp$ to $\pbs$.  
\hfill $\Box$ 

\subsection{Proof of \eqref{equa:reste4} in Lemma~\ref{lemm:expansion}}
\label{sect:asymptotique2}
For all $x \in \RR_n $, one can write
\begin{align}
\label{equa:U}
\log    U_i(\ab,n,x)    &=    \Big[   \frac{    \xi^+(n,x)}{\xi(n-1,x)}
-a_{\ii(\ab,n,x)} \Big] \Big[ \log \frac{a_i}{a_{\ii(\ab,n,x)}} - \log
\frac{1-a_i}{1-a_{\ii(\ab,n,x)}} \Big]\\
\nn & \qquad - \KL\big(a_{\ii(\ab,n,x)} |a_i\big). 
\end{align}
Let $c_0$ be the quantity defined as
\begin{equation}
\label{equa:c0}
c_0 = \inf_{\ab \in \Theta_\ab} \inf_{i \neq j} \KL(a_i|a_j)>0.
\end{equation}
where  the  strict  inequality  comes from~\eqref{equa:eps0}  and  the
continuity of $\KL(\cdot|\cdot)$. 
Recall       $\eps_0$      in~\eqref{equa:eps0}       and      $\eps'$
in~\eqref{equa:choix_eps}, and let $\eps''$ be the quantity defined as 
\begin{equation} \label{equa:choix_eps2}
\eps''=\min \Big\{ \eps', \frac {c_0} 4 \Big( \log \frac{1-
  \eps_0}{\eps_0} \Big)^{-1} \Big\}.
\end{equation}
 Let $\VV(\eps'')  \subset \Theta_\ab$  be the neighborhood  of $\abs$
 defined as 
\begin{equation}
  \label{equa:control}
\VV(\eps'')=\Big\{ \ab  \in \Theta_\ab \, :  \, \max \left\{  \| \ab -
  \abs \|, \| \betab - \betas \| \right\} \leq \eps'' \Big\}.  
\end{equation}
Fix $0< \delta <1$.  Let $\ab$ be in $\VV(\eps'')$ and assume that
$A_n^\delta(\eps'')$  defined by~\eqref{equa:eventA} occurs,  then for
all $x \in \Rn^\delta $ and any $i \neq \ii(\ab,n,x)$ 
\begin{equation}
  \label{equa:controlU}
  \log U_i(\ab,n,x) \leq - \frac{c_0}{2}.
\end{equation}
Indeed, using~\eqref{equa:control} and the fact that
$A_n^\delta(\eps'')$ occurs with $\eps'' \leq \eps'$, we have for all 
$x \in \RR_n^\delta$ 
\begin{equation}
  \label{equa:iChapeau}
  \ii(\ab,n,x)=\ii(\abs,n,x) \quad \mbox{and} \quad 
  \omega_x = a_{\ii(\ab,n,x)}^\star=a_{\ii(\abs,n,x)}^\star. 
\end{equation}
Then, if $i \neq \ii(\ab,n,x)$,
using~\eqref{equa:U},~\eqref{equa:c0},~\eqref{equa:iChapeau}, and the fact that 
\[
\Big| \log 
\frac{a_i}{a_{\ii(\ab,n,x)}}  -  \log \frac{1-a_i}{1-a_{\ii(\ab,n,x)}}
\Big| \leq 2 \log \frac{1- \eps_0}{\eps_0},
\]
yield
\begin{equation*}
  \log U_i(\ab,n,x) \leq 2 \log \frac{1- \eps_0}{\eps_0} \left(\Big| h
\big(\frac{ \xi^+(n,x)}{\xi^-(n,x)}\big) - h\big( \frac 1 {\rho_x}
\big) \Big| + \|\ab -\abs \| \right) - c_0
\end{equation*}
with 
\[
h(u)= \frac{u}{1+u}. 
\]
Using the fact that  $0 \leq h'(u) \leq 1$, for any  $u \geq 0$, and the
fact that $A_n^\delta(\eps'')$ occurs, we have
\[
\log U_i(\ab,n,x) \leq 2  \log \frac{1- \eps_0}{\eps_0} \left(\eps'' +
  \|\ab -\abs \| \right) - c_0.
\]
Using   our  choice  of   $\ab$  and   $\eps''$  achieves   the  proof
of~\eqref{equa:controlU}. 
Now,     from~\eqref{equa:controlU},     we     deduce     that     if
$A_n^\delta(\eps'')$ occurs, 
\begin{equation}
\label{equa:majo_r_n_theta}
0 \leq r_n(\t) \leq \left| \RR_n \setminus \RR_n^\delta \right| \cdot \log \Big( 1
+ \frac{d-1}{\eps_0} \Big) + \sum_{x \in \RR_n^\delta} \log \Big( 1 
+ \frac{d-1}{\eps_0} \ee^{-\xi(n-1,x) c_0/2 }\Big).
\end{equation}
Assume that the event $E^\delta_n$ defined by
\begin{equation} \label{equa:eventE}
  E^\delta_n=\left\{  \forall x \in  \Rn^\delta \,  : \,  \xi(n,x) \geq
    n^{\delta/2} 
\right\}, 
\end{equation}
occurs. Then
\begin{equation}
\label{equa:majo_r_n_theta2}
0 \leq r_n(\t) \leq \frac{d-1}{\eps_0} \Big( \left| \RR_n \setminus \RR_n^\delta \right|
 + |\RR_n^\delta| \ee^{-n^{\delta/2}c_0/2} \Big).
\end{equation}
From Lemma~\ref{lemm:Z},~\ref{lemm:xi_omega} and~\ref{lemm:excursion}, we conclude
that~\eqref{equa:reste4} occurs. \hfill $\Box$

\subsection{Convergence of  $\widehat \pb_n$} From  the definitions of
$\widehat \t_n$ and $\baa$ respectively given 
by~\eqref{eq:estimator} and~\eqref{eq:estimatorPMLE}, we have 
\begin{equation}
\label{equa:positivite}
\ell_n(\widehat \t_n) - \ell_n(\btet) \geq 0 \quad \mbox{and} \quad
L_n(\widehat \ab_n) - L_n(\baa) \leq 0.
\end{equation}
Recalling~\eqref{equa:expansion}, \eqref{equa:positivite} implies
\begin{equation}
\label{equa:positivite2}
K_n(\widehat \t_n) - K_n(\btet) +  \frac 1 {R_n} [r_n(\widehat \t_n) -
r_n(\btet)]\geq 0.
\end{equation}
From  Theorem~\ref{theo:cvsupport},  when $n$  is  large enough,  both
$\widehat  \ab_n$  and   $\baa$  belong  to  $\VV(\eps'')$  introduced
in~\eqref{equa:control}  with large  probability.   Assuming that  the
event  $A_n^\delta(\eps'')$ defined by~\eqref{equa:eventA}  occurs, we
have 
\[
\ii(\baa,n,x)=   \ii(\widehat   \ab_n,n,x)   =  \ii(\abs,n,x),   \quad
\mbox{for all $x \in \RR_n^\delta$}, 
\]
and
\[
K_n(\widehat   \t_n)  \leq  \log   \Big(\frac{1-\eps_0}{\eps_0}  \Big)
\frac{|\RR_n \setminus \RR_n^\delta|}{R_n} + \frac 1 {R_n} \sum_{x \in
  \RR_n} \log \widehat p_{\ii(\abs,n,x)}, 
\]
as well as  
\[
K_n(\btet)  \geq \log \Big(\frac{\eps_0}{1-\eps_0}  \Big) \frac{|\RR_n
  \setminus  \RR_n^\delta|}{R_n} +  \frac 1  {R_n} \sum_{x  \in \RR_n}
\log \bar p_{\ii(\abs,n,x)}, 
\]
which holds for large $n$ when $\bpp$ is in $[\eps_0,1-\eps_0]$, 
and finally
\begin{equation}
\label{equa:majo_K}
K_n(\widehat      \t_n)     -      K_n(\btet)     \leq      2     \log
\Big(\frac{1-\eps_0}{\eps_0}      \Big)     \frac{|\RR_n     \setminus
  \RR_n^\delta|}{R_n} - \KL(\bpp|\widehat \pb_n). 
\end{equation}
Assuming furthermore  that $E_n^\delta$ defined by~\eqref{equa:eventE}
occurs, then~\eqref{equa:majo_r_n_theta2} occurs. 
All in all,
combining~\eqref{equa:majo_r_n_theta2},~\eqref{equa:positivite2}
and~\eqref{equa:majo_K}  yields the  existence of  a positive
constant $C$, depending on $\eps_0$ and $d$ only, such that 
\[
\KL(\bpp|\widehat  \pb_n)  \leq   C  \Big(\frac{|\RR_n
  \setminus     \RR_n^\delta|}{R_n}    +    \frac{|\RR_n^\delta|}{R_n}
\ee^{-n^{\delta/2}c_0/2}\Big),
\]
with  large  probability.  From  the  fact that  $\bpp$  converges  in
$\PPs$-probability  to $\pbs$,  the continuity  of $\KL(\cdot|\cdot)$,
Lemma~\ref{lemm:Z},~\ref{lemm:xi_omega}  and~\ref{lemm:excursion},  we
conclude    that   $\widehat   \pb_n$    converges   to    $\pbs$   in
$\PPs$-probability. \hfill $\Box$ 

%
%
%
%
%

\subsection{Intermediate lemmas}

\begin{lemm}   \label{lemm:Z}   There   exists   a   random   variable
$Z(\delta)\geq 0$ such that
\begin{align} 
  \label{eq:zd}
  &\frac{\displaystyle \big| \Rn \setminus \Rn^{\delta}\big|}{\log^2 n }
  \xrightarrow[n \to \infty]{\mbox{{\small law}}}
  Z(\delta), \quad \mbox{with}\\
  \label{eq:zd0}
  &Z(\delta)   \xrightarrow[\delta    \to   0]{}   
  0    \quad   \mbox{in $\PPs$-probability}.
\end{align}
\end{lemm}

\begin{proof} By the invariance principle,
  \begin{equation}
    \label{eq:invpr}
    \left(\frac{V([u  \log^2   n])}{\log   n};  u   \geq  0   \right)
    \xrightarrow[n \to \infty]{\mbox{{\small law}}} (W(u); u \geq 0),
  \end{equation}
  with  $W$ a  standard  Brownian  motion. Recall  the  definition of  $
  c_n^\delta$, $ \DD_n^\delta=]b_n,c_n^\delta]$ 
  with $ \DD_n^\delta$ from (\ref{equa:D_n}), and define
  \begin{align*}
    c_\8=  \inf\big\{ u& >0:  W(u)-\min_{v \in [0,u]}W(v)  \geq 1\big\},
    \qquad b_\8= \argmin_{v \in [0,c_\8]}W(v),\\
    &c_\8^\delta=  \min\big\{ u >b_\8:  W(u)-W(b_\8)  \geq 1- \delta \big\}.
  \end{align*}
  By  (\ref{eq:invpr})  and from  well-known  results  on RWRE  (e.g.,
  \cite[Sect. 2.5]{ZeitouniSF}), we have the joint convergence of 
  \[  \frac{b_n}{\log^2   n}  \xrightarrow[n  \to  \infty]{\mbox{{\small
        law}}} b_\8,\qquad 
  \frac{c_n^\delta }{\log^2   n}  \xrightarrow[n  \to  \infty]{\mbox{{\small
        law}}} c_\8^\delta,
  \qquad \frac{\max   \Rn}{\log^2  n}  \xrightarrow[n  \to
  \infty]{\mbox{{\small law}}} c_\8.
  \]
  Then, the  convergence in (\ref{eq:zd}) follows, with  the limit given
  by
  \begin{align*}
    Z(\delta) 
    &={\rm Leb} \Big( \big\{ u<b_\8  \,:\, \max_{v \in [u,b_\8]} W(v) - W(u)
    \leq \delta \big\} \cup [c_\8^\delta, c_\8] \Big)\\
    &=Z_1(\delta)+ Z_2(\delta) 
  \end{align*}
  with   $ {\rm Leb}(A) $ the  Lebesgue measure of a
  Borel  set $A$  and  $Z_2(\delta) =  c_\8-c_\8^\delta$.   It is  not
  difficult to see  that $ c_\8^\delta \nearrow c_\8$  a.s. as $\delta
  \searrow 0$, and then $Z_2(\delta)$ vanishes.
  Let us prove in details that $Z_1(\delta)$ vanishes. Letting
  \[
  A_\delta = \big\{ u<b_\8  \,:\, \max_{v \in [u,b_\8]} W(v) - W(u)
  \leq \delta \big\},
  \]
  we have, as $\delta \searrow 0$,
  \[
  A_\delta \searrow A_{0^+} \subset A'=
  \big\{ u<b_\8  \,:\,  W(u) \geq  W(v) , v \in [u,b_\8]\big\} .
  \] 
  By Fubini, we compute
  \begin{align*}
    \EEs {\rm Leb} (A') &= \int_0^\8 \PPs( u < b_\8, u \in A') du ,
  \end{align*}
  and we show that the integrand is zero. For all $u \geq 0$,
  \begin{align*}
    \PPs( u  < b_\8, u  \in A') &=  \lim_{\alpha \searrow 0} \PPs(  u <
    b_\8, u \in A', u \leq b_\8 + \alpha)\\ 
    &\leq  \limsup_{\alpha \searrow  0} \PPs(  W(u)  \geq W(v)  , v  \in
    [u,u+\alpha] ), 
  \end{align*}
  and  the  last  probability   is  zero  by  Iterated  Logarithm  law
  \cite[Th. 9.12]{KaratzasShreeve}). 
  Finally,  $\lim_{\delta \to 0}  {\rm Leb}(A_\delta)=0$  a.s., ending
  the proof of (\ref{eq:zd0}).  
\end{proof}


\begin{lemm} \label{lemm:xi_omega} 
Let Assumptions~\ref{as:recvar} and~\ref{as:ell} hold, let $\eps'$ and $\eps''$
be   such   that   \eqref{equa:choix_eps} and \eqref{equa:choix_eps2}   are  satisfied,   and   let
$A^\delta_n(\eps'')$ and $A^\delta_n(\eps'')$  be  the  events defined  by  \eqref{equa:eventA}. 
Then the following convergence holds 
\begin{equation} \label{equa:prob_A_complement}
 \PPs   \Big( {}^{\complement}A^\delta_n(\eps')   \Big), \ \PPs   \Big( {}^{\complement}A^\delta_n(\eps'')   \Big)
  \xrightarrow[n \to \infty]{} 0, \quad \mbox{for any $\delta$ in $(0,1)$}.
\end{equation}
\end{lemm}

\begin{proof}
Under  $P_\omega(\cdot  |  \xi_(n-1,x)=m)$,  the pair  $(  \xi^+(n,x),
\xi^-(n,x))$ is distributed  as $({\mathrm B}(m, \omega_x), m-{\mathrm
  B}(m,  \omega_x))$,  where $\mathrm{B}(m,q)$  is  a binomial  random
variable with  sample size $m$  and probability of success  $q$. Using
\citeauthor{Hoeffding}'s concentration inequality yields 
\begin{equation}\label{eq:devloctime} 
  P_\omega \left( \Big\vert \frac{\xi^+(n,x)} {m} - \omega_x \Big\vert
    >   \alpha   \,  \Big|   \,   \xi(n-1,x)=m   \right)  \leq   2\exp
  [-2m\alpha^2], \quad \forall \alpha > 0.
\end{equation}

Assume that the event $\{\xi(n-1,x)= {m}\}$ occurs. Then, we have
\[ 
  \Big\vert \frac{\xi^+(n,x)} { \xi^-(n,x)} - \frac{1} {\rho_x} \Big\vert 
  = 
  \Big\vert f \Big(\frac{\xi^+(n,x)} {m} \Big) -f(\w_x) \Big\vert
  \leq \sup_{I} f' \cdot \Big\vert \frac{\xi^+(n,x)} {m} - \w_x \Big\vert,
\]
with 
\[
  f(u)=        \frac{u}{1-u}       \quad        \mbox{and}       \quad
  I=\left[   \frac{\xi^+(n,x)}{m}\wedge   \omega_x,  \frac{\xi^+(n,x)}
    {m} \vee \omega_x \right]. 
\]
Under   Assumption  \ref{as:ell},   uniform  ellipticity   occurs  and
\eqref{eq:devloctime} implies that we can find a constant $K>0$
depending on $\eps''$ only such that
\begin{equation} \label{equa:concentration}
  P_\omega \left( \Big\vert \frac{\xi^+(n,x)} { \xi^-(n,x)} - \frac{1}
    {\rho_x} \Big\vert  > \eps'' \,  \Big| \, \xi(n-1,x)=m  \right) \leq
  K^{-1} \exp[-K m].
\end{equation}
Recall the event $E^\delta_n$ defined by~\eqref{equa:eventE} 
and denote its complement by $^{\complement}E^\delta_n$.  We have
\[
  P_\omega\left( {}^{\complement}A^\delta_n(\eps'')  \right) 
  \leq 
  P_\omega\left( {}^{\complement}A^\delta_n(\eps''), E^\delta_n\right) 
 + P_\omega\left(^{\complement}E^\delta_n\right). 
\]
Now,
\[
  P_\omega\left( {}^{\complement}A^\delta_n(\eps''), E^\delta_n\right)
  \leq    \sum_{x    \in    \Rn^\delta}    P_\omega\left(    \Big\vert
    \frac{\xi^+(n,x)}  { \xi^-(n,x)} -  \frac{1} {\rho_x}  \Big\vert >
    \eps'', \xi(n-1,x) \geq n^{\delta/2} \right),
\]
and writing $\left\{\xi(n-1,x) \geq n^{\delta/2} \right\} = \bigcup_{m\geq
  n^{\delta/2}   } \{ \xi(n-1,x)=m\}$  and  using~\eqref{equa:concentration}
yield
\[
  P_\omega\left( {}^{\complement}A^\delta_n(\eps''), E^\delta_n\right)
  \leq
  \sum_{x \in  \Rn^\delta} \sum_{m \geq n^{\delta/2}  } K^{-1} \exp[-K
  m] \cdot P_\omega(\xi(n-1,x)=m).
\]
Using the fact that 
\[
 \exp[-K m] \leq \exp[-K n^{\delta/2} ], \quad \mbox{for any
$m \geq n^{\delta/2}$},
\]
the fact that 
\[
  \sum_{m \geq n^{\delta/2} } P_\omega(\xi(n-1,x)=m) \leq 1,
\]
and the fact that $ |\Rn^\delta| \leq c_n$, yield
\begin{equation} \label{equa:interQuenched}
  P_\omega\left( {}^{\complement}A^\delta_n(\eps''), E^\delta_n\right)
  \leq K^{-1} c_n \exp[-K n^{\delta/2} ].
\end{equation}
Then, combining~\eqref{equa:interQuenched} and 
\[
\PPs\left( {}^{\complement}A^\delta_n(\eps''), E^\delta_n\right) 
= \Es
\left(        P_\omega\left(       {}^{\complement}A^\delta_n(\eps''),
    E^\delta_n\right)  \right)
\]
yields 
\[
\PPs\left( {}^{\complement}A^\delta_n(\eps''), E^\delta_n\right)  
\leq  
K^{-1}  \exp[-K  n^{\delta/2}  ] \log^3 n + \Ps(c_n > \log^3 n),
\]
and   from   the  fact   that   $c_n   /   \log^2  n$   converges   in
$\Ps$-distribution to $c_\infty$, we deduce that 
\[
  \PPs\left(   {}^{\complement}A^\delta_n(\eps''),   E^\delta_n\right)
  \xrightarrow[n  \to \infty]{}  0,  \quad \mbox{for  any $\delta$  in
    $(0,1)$}.
\]
From Lemma~\ref{lemm:excursion} below, we have
\[
  \PPs\left(^{\complement}E^\delta_n\right)=                        \Es
  \left(P_\omega\left(^{\complement}E^\delta_n\right) \right) 
  \xrightarrow[n \to \infty]{} 0, \quad \mbox{for any $\delta$ in $(0,1)$},
\]
and this achieves the proof of \eqref{equa:prob_A_complement} since $A^\delta_n(\eps'') \subset A^\delta_n(\eps')$.
\end{proof}


\begin{lemm} \label{lemm:excursion}
Let Assumptions~\ref{as:recvar} and~\ref{as:ell} hold   and   let
$E^\delta_n$  be  the  event  defined  by~\eqref{equa:eventE}.
Then the following convergence holds
\begin{equation} \label{equa:prob_E_complement}
  \PPs\left(^{\complement}E^\delta_n\right)     \xrightarrow[n     \to
  \infty]{} 0, \quad \mbox{for any $\delta$ in $(0,1)$}. 
\end{equation}
\end{lemm}

\begin{proof} 

Note that
\[
  \PPs\left(^{\complement}E^\delta_n\right) \leq  \PPs\left( \exists x
    \in  \GG_n^\delta  \,  :  \,  \xi(n,x) <  n^{\delta/2}  \right)  +
  \PPs\left(  \exists   x  \in  \DD_n^\delta   \,  :  \,   \xi(n,x)  <
    n^{\delta/2} \right) ,
\]
where  $\GG_n^\delta$  and   $\DD_n^\delta$  are  respectively  defined
by~\eqref{equa:G_n} and~\eqref{equa:D_n}. We first show that 
\begin{equation}
  \label{equa:chernovDn}
  \PPs\left(\exists x  \in
    \DD_n^\delta \, : \, \xi(n,x) < n^{\delta/2}\right) \xrightarrow[n
  \to   \infty]{} 0, \quad \mbox{for any $\delta$ in $(0,1)$}.
\end{equation}
Let  $x$ be  in $\DD_n^\delta$.  Define  $T(y)=\inf\{t\geq 1  \, :  \,
X_t=y\}$ the  first hitting time  of~$y$. The probability  of visiting
$x>b_n$ during a given excursion 
from $b_n$ to $b_n$ is (e.g., \cite[formula (2.1.4)]{ZeitouniSF})
\begin{equation} \label{equa:visite} 
  P_\w(T(x) < T(b_n) \, | \, X_0 =
b_n)    =    \frac{\w_{b_n}}{\sum_{j=b_n}^x    \exp[V(j)-V(b_n)]}    >
\frac{\eps_0}{c_n} \cdot n^{\delta-1},
\end{equation}
where  the lower bound  follows from  \eqref{eq:assumpa} and  the fact
that $x$ belongs to $\DD_n^\delta$.


Fix $\eps>0$ and let $k_n$ denote the number of excursions of $X$ from
$b_n$ to $b_n$ before time $n$.  From (2.14) in \cite{GPS}, we have
\[ P_{\omega}  \Big( \Big|\frac  {k_n}n - \frac1{\gamma_n}  \Big| \geq
\eps  \Big)  \xrightarrow[n \to  \infty]{}  0,  \quad \mbox{for  \quad
$\Ps$-a.a. $\w$},
\]
where $\gamma_n$ is  the average length of an  excursion from $b_n$ to
$b_n$. Let $F_n$ denote  the event $F_n=\left\{ k_n\geq n \cdot \gamma_n^{-1}
/2\right\}$.
We have
\[ 
  P_\omega \left( \exists x \in \DD_n^\delta \, : \, \xi(n,x) <
n^{\delta/2}, \, F_n \right)
\leq
\sum_{x \in \DD_n^\delta}  P_\omega \left( \xi(n,x) < n^{\delta/2},
\, F_n \right).
\]
Using the independence of excursions from $b_n$ to $b_n$ yields
\[ 
  P_\omega \left( \xi(n,x) < n^{\delta/2}, \, F_n \right) 
  \leq 
  \mathrm{Prob}  \big( \mathrm{B}\left(n\cdot  \gamma_n^{-1}/2, \eps_0
    c_n^{-1} n^{\delta-1} \right) \leq n^{\delta/2}\big),
\]
where $\mathrm{B}(m,q)$ is a binomial random variable with sample size
$m$ and probability of success $q$. Combining the Chernov's bound
\[
\mathrm{Prob} \big( \mathrm{B} (m,q) \leq m(q-\alpha)\big) \leq \exp
  \big[ -m \KL(q-\alpha|q) \big].
\]
with  $m=n\cdot  \gamma_n^{-1}/2$, $q=\eps_0c_n^{-1}n^{\delta-1}$  and
$m(q-\alpha)  = n^{\delta/2}$,  the  fact that  $ |\DD_n^\delta|  \leq
c_n$, that
\[
  \gamma_n = \sum_{x=0}^{c_n} \frac{\mu(x)}{\mu(b_n)} \leq 2(1+c_n),
\]
and that
\[
\frac{c_n }{\log^2   n}  \xrightarrow[n  \to  \infty]{\mbox{{\small
        law}}} c_\8,
\]
yield~\eqref{equa:chernovDn}. 

Now, we turn to the case where
$x$ is  in $\GG_n^\delta$,  and a different  argument is  needed. From
Lemma 1 in \cite{Golosov}, 
\[
  \PPs(T(b_n)> n) \xrightarrow[n \to \infty]{} 0,
\]
hence,
\[
  \PPs(\xi(n,b_n)=0) \xrightarrow[n \to \infty]{} 0.
\]
Note that
\[
  \big\{ \exists x \in \GG_n^\delta  \, : \, \xi(n,x)=0\big \} \subset
  \big \{ \xi(n,b_n)=0 \big \}. 
\]
The probability of visiting $b_n$ during a given excursion from $x$ to $x$ is
\begin{equation} \label{equa:visite2} 
  P_\w(T(b_n) < T(x) \, | \, X_0 = x) = \frac{\w_{x}}{\sum_{j=x}^{b_n}
    \exp[V(j)-V(x)]} < n^{-\delta},
\end{equation}
where  the  larger  bound  follows   the  fact  that  $x$  belongs  to
$\GG_n^\delta$.  Let $h_n(x)$ denote the number of 
returns to $x$ before reaching $b_n$. From \eqref{equa:visite2}, this 
is   a  geometric   variable  with   success  probability   less  than
$n^{-\delta}$.  Now, 
\begin{align*} 
  P_\omega  \left(\exists x  \in \GG_n^\delta  \, :  \,  \xi(n,x) \leq
    n^{\delta/2} \right) 
  \leq  P_\omega \left(\exists  x \in  \GG_n^\delta \,  : \,  1 \leq
    \xi(n,x) \leq n^{\delta/2} \right)&\\
  + P_\omega \left(\exists x \in \GG_n^\delta \, : \,
    \xi(n,x)=0 \right)& 
\end{align*}
Hence,
\begin{align*}
  P_\omega \left(\exists x \in \GG_n^\delta \, : \, 1 \leq \xi(n,x) \leq
  n^{\delta/2} \right)
  &\leq
  \sum_{x  \in \GG_n^\delta}  P_\omega \left(\,  1 \leq  \xi(n,x) \leq
    n^{\delta/2} \right) \\
  &\leq
  \sum_{x  \in \GG_n^\delta}  P_\omega \left(\,  1 \leq  h_n(x) \leq
    n^{\delta/2} \right)\\
  &\leq |\GG_n^\delta| \cdot [1 - (1-n^{-\delta})^{n^{\delta/2}}]\\
  &\leq c_n \cdot [1 - (1-n^{-\delta})^{n^{\delta/2}}].
\end{align*}
Hence,
\[
\PPs\left(\exists x \in \GG_n^\delta \, : \, 1 \leq \xi(n,x) \leq
  n^{\delta/2}  \right)   \leq  [1  -  (1-n^{-\delta})^{n^{\delta/2}}]
\log^3 n + \Ps(c_n > \log^3 n), 
\]
which  combined with  the fact  that $  c_n /\log^2  n $  converges in
distribution to $c_\8$ proves that
\[
\PPs\left( \exists x
    \in \GG_n^\delta \, : \, \xi(n,x) < n^{\delta/2} \right). \qedhere
\]
\end{proof}

\section{Examples} \label{sect:ex}


\subsection{Particular Case: recurrent Temkin model} \label{sect:ex1}

\begin{exam} \label{ex:Temkin} 
Let  $\eta  =  \frac12  \delta_{a}  +
\frac12 \delta_{1-a}$.  Here, the unknown parameter is  the support $a
\in \Theta \subset (0,1/2) $ (namely~$\theta=a$).
\end{exam}

\begin{prop} \label{prop:Temkin} In the framework of 
Example~\ref{ex:Temkin},   the   limiting   function  $L_\8$   defined
by~\eqref{eq:Linfini} is deterministic and given by
\[ 
 L_\8(a) = - \big[ H(a^\star) + \KL(a^\star | a) \big].
\]
\end{prop}

\begin{proof} Note that
\begin{equation} \label{equa:omega2points} \{\widetilde \omega (x), 1-
\widetilde \omega (x)\} =  \{ a^\star,1-a^\star\}, \quad \mbox{for any
$x \in \Z$}.
\end{equation}
Recalling  that $\nu^+(x)  =  \nu(x) \cdot  \widetilde \omega(x)$  and
$\nu^-(x) = \nu(x) \cdot [1-\widetilde \omega(x)]$ and noting that
\[ 
    \max_{u \in \{a,1-a\}}
\Big\{ a^\star \log u + (1-a^\star)\log (1-u) \Big\} = 
a^\star \log a + (1-a^\star) \log (1 - a),
\]
implies that~\eqref{eq:Linfini} might be rewritten as
\[ 
 L_\8(a) = \Big[ a^\star \log a + (1-a^\star) \log (1 - a) \Big]
\sum_x \nu(x).
\]
The fact that $\sum_{x \in \Z} \nu(x)=1$ achieves the proof.
\end{proof}


\subsection{Case of two-point support}
\begin{exam} \label{ex:2points} 
Let  $\eta =  p_1 \delta_{a_1}  + p_2
\delta_{a_2}$ with $a_1 <1/2<a_2$, and
\begin{equation}
\nn       p_1=\frac{\log       \frac{1-a_1}{a_1}}{\log
\frac{a_2(1-a_1)}{a_1(1-a_2)}},          \quad          p_2=\frac{\log
\frac{a_2}{1-a_2}}{\log \frac{a_2(1-a_1)}{a_1(1-a_2)}}.
\end{equation}
Here, the unknown parameter is the support $\ab=(a_1,a_2)$.
\end{exam}

The following result extends Proposition \ref{prop:Temkin}.

\begin{prop} \label{prop:2points} 
  In the framework of Example~\ref{ex:2points}, the limiting function
   $ L_\8$ defined by~\eqref{eq:Linfini} is deterministic and given by
\[ 
 L_\8(\ab) = - \frac{1}{a_2^\star-a_1^\star} \sum_{j=1}^2
\Big\vert   a_j^\star-\frac{1}{2}\Big\vert    \big[   H(a_j^\star)   +
\min_{i=1,2} \KL(a_j^\star | a_i) \big].
\]
\end{prop}

\begin{proof} 
By \eqref{eq:Lentropy},
\[
  L_\8(\ab) = - \sum_{j=1}^2 \nu^{(j)} \cdot \big[ H(a_j^\star) +
  \min_{i=1,2} \KL(a_j^\star |  a_i) \big], \qquad \nu^{(j)}= \sum_{x:
    \widetilde \omega (x)=a^\star_j} \nu(x). 
\]
So,  the  formula  for  $L_\8(\ab)$  follows from  the  values  of  the
coefficients $\nu^{(j)}$: these are determined by
\[
  \sum_{j=1}^2 \nu^{(j)}=1, \qquad \sum_{j=1}^2 \nu^{(j)} \cdot 
  (2a_j^\star-1)=0, 
\]
where the  second equality means  that the mean  drift is zero  in the
recurrent  case. Indeed,  the  drift per  time  unit on  time
interval  $[0,n]$,  $n^{-1}  \sum_{t=0}^{n-1}  (2  \w_{X_t}-1)  =
\sum_x \nu_{n}(x) (2 \widetilde \omega  (x) -1)$, vanishes as $n \to
\8$  by the  law  of large  numbers  for martingales.  For a  direct
derivation,    we    can     write,    from    \eqref{equa:Nu2}    and
\eqref{equa:tildeOmega},
\begin{eqnarray}\nn   \sum_x  \nu(x)   \widetilde   \omega(x)  =\sum_x
\nu^+(x) =\sum_x \nu^-(x)= \sum_x \nu(x) (1-\widetilde \omega(x)),
\end{eqnarray} and then  $ \sum_x \nu(x) (2 \widetilde  \omega (x) -1)
=0, $ which,  after reorganizing the terms, gives  the second equality
above. 
\end{proof}


\subsection{Particular      Case:      recurrent      lazy      Temkin
model} \label{sect:lazy} 

We  present a  simple  example  for which  $L_\8(\cdot)$  is a  random
variable.

\begin{exam} \label{ex:lazy} 
Let $\eta = \frac{1-r}{2} \delta_{a} +
r  \delta_{1/2} +\frac{1-r}{2}  \delta_{1-a}  $ with  $a  <1/2, r  \in
(0,1)$.  Here, we have $\ab=(a,\frac 1 2,1-a)$ and
$\pb=(\frac{1-r}{2},r,\frac{1-r}{2})$, and the unknown parameter is $(a,r)$.
\end{exam} 
\begin{prop} 
\label{prop:lazy} 
In the framework of 
Example~\ref{ex:lazy}, there exists $a' \in (0,a^\star)$, depending
only  on  $a^\star$,  such   that  the  limiting  function  $ L_\8(\ab)$ is
deterministic when $a \in (0,a']$ and random when $a \in
(a',1/2)$: 
\begin{equation*}
  {\Var}^\star( L_\8(\ab) )>0.
\end{equation*}
\end{prop}

\begin{proof} 
Set $\uab=\{a,1/2,1-a\}$. By \eqref{eq:Lentropy},
\begin{align*}
 L_\8(\ab) = &- \nu^{(a^\star)} \cdot \big[ H(a^\star) +
\min_{\uab} \KL(a^\star | \cdot) \big] \\
&\qquad- \nu^{(1-a^\star)} \cdot 
\big[  H(1-a^\star)  +  \min_{\uab}  \KL(1-a^\star |  \cdot)  \big]  -
\nu^{(1/2)} \cdot H(1/2), 
\end{align*}
with
\[ \nu^{(u)}= \sum_{x:  \widetilde \omega (x)=u} \nu(x), \quad
u \in \uab.
\]
So,  the  formula  for  $L_\8(\ab)$  follows from  the  values  of  the
coefficients $\nu^{(u)}$: these satisfy
\[ \sum_{u  \in \uab} \nu^{(u)}=1, \qquad  \sum_{u \in \uab}
\nu^{(u)} (2 u -1)=0,
\]
where  the  second  equality  is  equivalent to  $  \nu^{(a^\star)}  =
\nu^{(1-a^\star)}$. Hence, we have  $ 2\nu^{(a^\star)} + \nu^{(1/2)} =
1$. Furthermore, noting that
\[   H(a^\star)=H(1-a^\star)   \quad   \mbox{and}   \quad   \min_{\uab}
\KL(a^\star | \cdot) =\min_{\uab} \KL(1-a^\star | \cdot)
\]
yields
\begin{equation}
  \label{equa:L_Lazy}
  L_\8(\ab) = \nu^{(1/2)}  \big[ H(a^\star) + \min_{\uab} \KL(a^\star |
  \cdot) -  H(1/2) \big] -  \big[ H(a^\star) + \min_{\uab}  \KL(a^\star |
  \cdot) \big].
\end{equation}
One can see that in~\eqref{equa:L_Lazy}, the quantities 
\[
\big[ H(a^\star) + \min_{\uab} \KL(a^\star | \cdot) - H(1/2) \big]
\quad \mbox{and} \quad
\min_{\uab} \KL(a^\star | \cdot) \big]
\]
are  deterministic.   Hence, $L_\8(\ab)$  is  random  if  and only  if
$\nu^{(1/2)}$   is  random   and  $\big[   H(a^\star)   +  \min_{\uab}
\KL(a^\star | \cdot) - H(1/2) \big]$ is non zero. 

The  facts that $\KL(a^\star  | \cdot)$  is continuous,  decreasing on
$(0,a^\star)$, increasing on $(a^\star,1)$, null at $a^\star$ and that
\[ \KL(a^\star  | 1/2)  = H(1/2) -  H(a^\star) \quad  \mbox{and} \quad
\KL(a^\star | a) \xrightarrow[a \to 0]{} + \infty,
\]
 yield the existence of a unique $a' \in (0,a^\star)$ such that
 \[ \KL(a^\star | a') = H(1/2) - H(a^\star),
 \]
 and as a consequence that for any $a \in (0,a']$,
\[ 
\min_{\uab} \KL(a^\star | \cdot) = \KL(a^\star | a') = H(1/2) - H(a^\star).
\]
Hence, the  limiting function $ L_\8(\ab)$ is  deterministic and equal
to $- H(1/2)=-\log 2$, for any $a \in (0,a']$. 
  
Now, let $a$ be in $(a',1/2)$. Recall that
\[ 
  \nu^{(1/2)}=  \sum_{x:  \widetilde  \omega  (x)=1/2}  \nu(x),  \quad
  \mbox{with}     \quad    \nu(x)     =     \frac{    \exp[-\widetilde
    V(x-1)]+\exp[-\widetilde V(x)]  }{2\sum_{z\in \Z} \exp[-\widetilde
    V(z)]}, 
\]
and that  $\widetilde \omega (x)=1/2 $  if and only if  $\widetilde V(x) =
\widetilde  V(x-1)$.  Now,  we  focus  on the  potential  of the  infinite
valley. Recall  that $\widetilde V(0)=0,$  $\{ \widetilde V(x)\}_{x \geq  0} $
and $\{ \widetilde V(-x) \}_{x  \geq 0}$ are two independent Markov chains
where  $\widetilde V(x)$  is  non-negative for  any  non-negative $x$  and
$\widetilde   V(x)$    is   positive    for   any   negative    $x$,   see
\cite{Golosov}. Especially, we have 
\[ 
  \Ps  \big(\widetilde V(x+1)=0  \,  | \,  \widetilde V(x)=0\big)=1  -
  \Ps \big(\widetilde V(x+1)= \log \rho^\star\, | \, \widetilde
  V(x)=0\big) =r,  
\]
where we set $\rho^\star=\frac{1-a^\star}{a^\star}$. Define the random variable $T
\geq 1$  as $ T  = \inf  \{ x \in  \N \, :  \, \widetilde V(x)  = \log
\rho^\star\}$. The random variable $T$ is 
distributed according  to a  geometric distribution with  parameter of
success $1-r$, and for any $0  \leq x \leq T-1$, the potential $\widetilde
V(x)$ is equal to $0$.  We have
\[ 
  \nu^{(1/2)}  =  \frac   {  \sum_{x<0}  \ee^{-\widetilde  V(x)}\1_{\widetilde
      V(x-1)=\widetilde    V(x)}    +(T-1)+\sum_{x=T}^{\infty}\ee^{-\widetilde
      V(x)}\1_{\widetilde V(x-1)=\widetilde V(x)}} 
  {\sum_{x\leq  0}  \ee^{-\widetilde  V(x)} +(T-1)  +  \sum_{x=T}^{\infty}
    \ee^{-\widetilde V(x)}} 
\]
Using the  definition of $(\widetilde  V(x) \,:\, x\in\Z)$ and  the strong
Markov  property, $T$  is independent  of $\FF:=\SS(\widetilde  V(x) \,:\,
x\leq 0 \mbox{ or } x \geq T)$, hence
\[ 
  \Var^\star[     \nu^{(1/2)}     |     \FF     ]     =     \Var^\star
  \left( \frac{T-1+\Upsilon'}{T-1+\Upsilon} \right), 
\]
with
\[ 
  \Upsilon'=\sum_{x<0;  x\geq T}  \ee^{-\widetilde  V(x)}\1_{\widetilde V(x-1)=\widetilde
    V(x)} 
  \quad \mbox{and} \quad 
  \Upsilon=\sum_{x\leq 0; x\geq T} \ee^{-\widetilde V(x)}.
\]
Since $\Upsilon  > \Upsilon',$ $\Var^\star[\nu^{(1/2)}|\FF]$  is a non
constant strictly positive random variable, hence
\[ 
  \Var^\star(\nu^{(1/2)})=\Es  \big( \Var^\star[\nu^{(1/2)}|\FF] \big)
  + \Var^\star( \Es[\nu^{(1/2)}|\FF] )>0. \qedhere 
\]
\end{proof}

\section{Numerical performance}
\label{sect:simus}

In this section, we explore the numerical performance of our 
estimation  procedure in  the  frameworks of  Examples~\ref{ex:Temkin}
to~\ref{ex:lazy}. We  compare our performance with  the performance of
the   estimator  proposed   by   \cite{AdEn}  in   the  framework   of
Example~\ref{ex:Temkin}. An explicit description of the form
of~\citeauthor{AdEn}'s  estimator  in   the  particular  case  of  the
one-dimensional nearest  neighbour path is provided in  Section 5.1 of
\cite{Comets_etal}. Therefore, one can  estimate  $\ts$ by  the  solution of  an
appropriate system of equations, as illustrated below.  

\paragraph*{Example~\ref{ex:Temkin}  (continued)}  In  this  case  the
parameter $\theta$ equals $a$ and we have 
\begin{align*} \label{eq:syst_eq_Temkin}
  v &= \Es[\omega_0] = \frac12 a^\star +
  \frac12 (1-a^\star)= \frac12,\\
  w &= \frac{\Es[\omega_0^2]}{\Es[\omega_0]} = \{\frac12 [a^\star]^2 + \frac12 (1-a^\star)^2\}
  \cdot v^{-1} = [a^\star]^2 + (1-a^\star)^2.
\end{align*}
Hence, among the visited sites, the proportion of those from which the
first move  is to the right (or  to the left) gives  no information on
the parameter $a^\star$. We need to push to the sites visited at least
twice to get some 
information. Among  those from which the  first move is  to the right,
the proportion of those from which 
the  second  move  is  also  to  the  right  gives  an  estimator  for
$[a^\star]^2 + (1-a^\star)^2$. Using this observation, we can estimate
$a^\star$.  
We  now   present  the  three   simulation  experiments  corresponding
respectively to Examples~\ref{ex:Temkin} to~\ref{ex:lazy}.  The
comparison with \citeauthor{AdEn}'s procedure is given only for 
Example~\ref{ex:Temkin}.   
For  each  of  the  three  simulations, we  \textit{a  priori}  fix  a
parameter  value $\ts$ as  given in  Table~\ref{tabl:theta_values} and
repeat 1,000 times the  procedure described below. 
\renewcommand{\arraystretch}{1.5}
\begin{table}[h]
  \centering
  \begin{tabular}{|c|c|}
    \hline
Simulation & Estimated parameter \\
    \hline
Example~\ref{ex:Temkin}  & $a^\star=0.3$ \\
    \hline
Example~\ref{ex:2points}  & $(a_1^\star, a_2^\star)=(0.4,0.7)$\\ 
    \hline
Example~\ref{ex:lazy}   & $(a^\star, r^\star)=(0.3,0.2)$ \\
    \hline
  \end{tabular}
  \caption{Parameter values for each experiment.}
  \label{tabl:theta_values}
\end{table}
We first generate a
random  environment  according  to  $\eta_\ts$  on  the  set  of  sites
$\{1, \dots, 10^5\}$.  In fact,  we do  not use the environment
values  for all the  sites, since  only few  of these
sites are  visited by the walk.  However the computation  cost is very
low comparing to the rest of the estimation procedure.  Then, we run a
random walk in this environment  and stop it successively at $n$, with
$n \in \{10^4\cdot k \, : \, 1\le k \le 10 \}$.  For each 
stop,  we  estimate  $\ts$   according  to  our  MPLE  procedure  (and
\citeauthor{AdEn}'s one for the first simulation).  In the
cases  of Examples~\ref{ex:Temkin}  and~\ref{ex:lazy},  the likelihood
optimization procedure  for the support  $a^\star$ was performed  as a
combination  of   golden  section  search   and  successive  parabolic
interpolation.    In  the   case  of   Example~\ref{ex:2points},  the
likelihood  optimization  procedure  for  $(a_1^\star,a_2^\star)$  was
performed according to the "L-BFGS-B" method 
of  \cite{ByrdEtAl} which  uses a  limited-memory modification  of the
``BFGS''  quasi-Newton  method  published  simultaneously in  1970  by
\citeauthor{Broyden}, \citeauthor{Fletcher}, \citeauthor{Goldfarb} and \citeauthor{Shanno}. 
In  all three cases,  the parameters  in Table~\ref{tabl:theta_values}
are  chosen such  that  the RWRE  is recurrent.

Figure~\ref{figu:compa} shows the  boxplots of our estimator
and \citeauthor{AdEn}'s  estimator obtained  from 1,000  iterations of
the  procedures in  Example~\ref{ex:Temkin}.  First,  we  shall notify
that in order to simplify the visualisation of the results, we removed
in  the boxplots corresponding  to Example~\ref{ex:Temkin}  about 13.5\% of
outliers values (outside  1.5 times the interquartile range above
the   upper  quartile  and   below  the   lower  quartile)   from  our
estimator.  Second,  we shall  also  notify  that  about  16\%  of
\citeauthor{AdEn}'s  estimation procedures could  not be  achieved for
the  simple reason  that  the  equation involving  $\abs$  had no solution.  We observe that our  procedure is far more
accurate than \citeauthor{AdEn}'s, and that the
accuracies of the procedure increase with the value of $n$. 

Figure~\ref{figu:boxplots_cas2}  shows the  boxplots of  our estimator
obtained    from    1,000    iterations    of   the    procedure    in
Example~\ref{ex:2points}.  Once again, we  shall notify
that in order to simplify the visualisation of the results, we removed
in  the boxplots corresponding  to Example~\ref{ex:2points}  about 13\% of
outliers values   from  both 
estimators of $a_1^\star$ and $a_2^\star$. We observe that the
accuracy of  the procedure  is high,  and that in  this case,  it does
increase with the value of $n$. 

Figure~\ref{figu:boxplots_cas3}  shows the  boxplots of  our estimator
obtained    from    1,000    iterations    of   the    procedure    in
Example~\ref{ex:lazy}. We  shall notify
that in order to simplify the visualisation of the results, we removed
in  the boxplots corresponding  to Example~\ref{ex:lazy}  about 12.3\% of
outliers values   from  our 
estimator of $a^\star$ and 2.8\% of
outliers values   from  our 
estimator of $r^\star$. We observe that the
accuracy of the procedure is high for the parameter of the support and
low for the probability  parameter, even if in both cases, the
accuracy increases with the value of $n$.

\begin{figure}[H]
 \centering
 \includegraphics[angle=270,width=12.5cm]{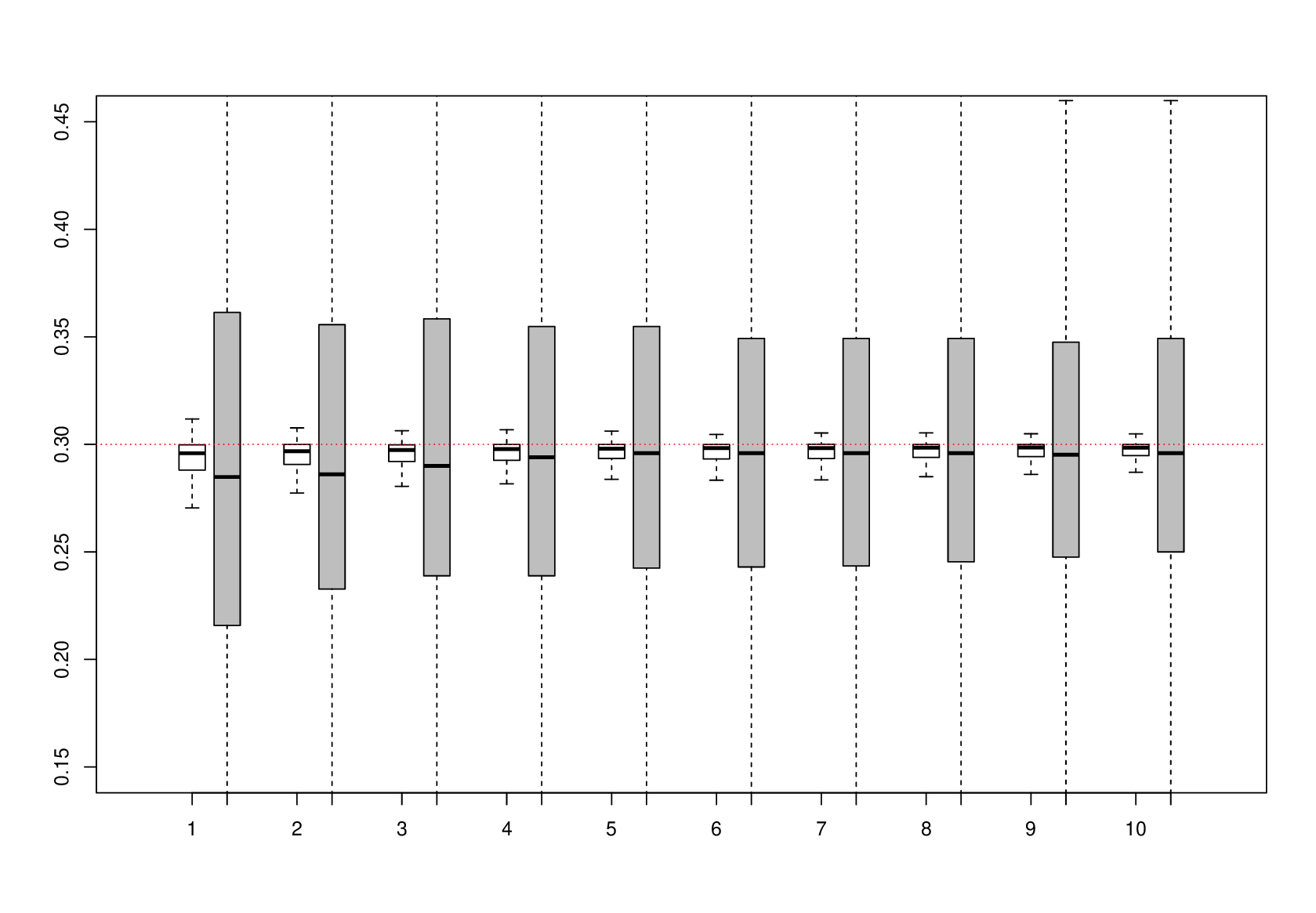} \\
 \caption{Boxplots   of   our   estimator   (left   and   white)   and
   \citeauthor{AdEn}'s estimator (right  and grey) obtained from 1,000
   iterations and for  values $n$ ranging in $ \{10^4 \cdot  k \, : \,
   1\le k  \le 10 \}$ ($x$-axis  indicates the value  $k$).  The panel
   displays  estimation of  $a^\star$ in  Example~\ref{ex:Temkin}. The
   true value is indicated by an horizontal line.}
 \label{figu:compa}
\end{figure}

\begin{figure}[H]
  \centering
  \includegraphics[angle=270,width=12.5cm]{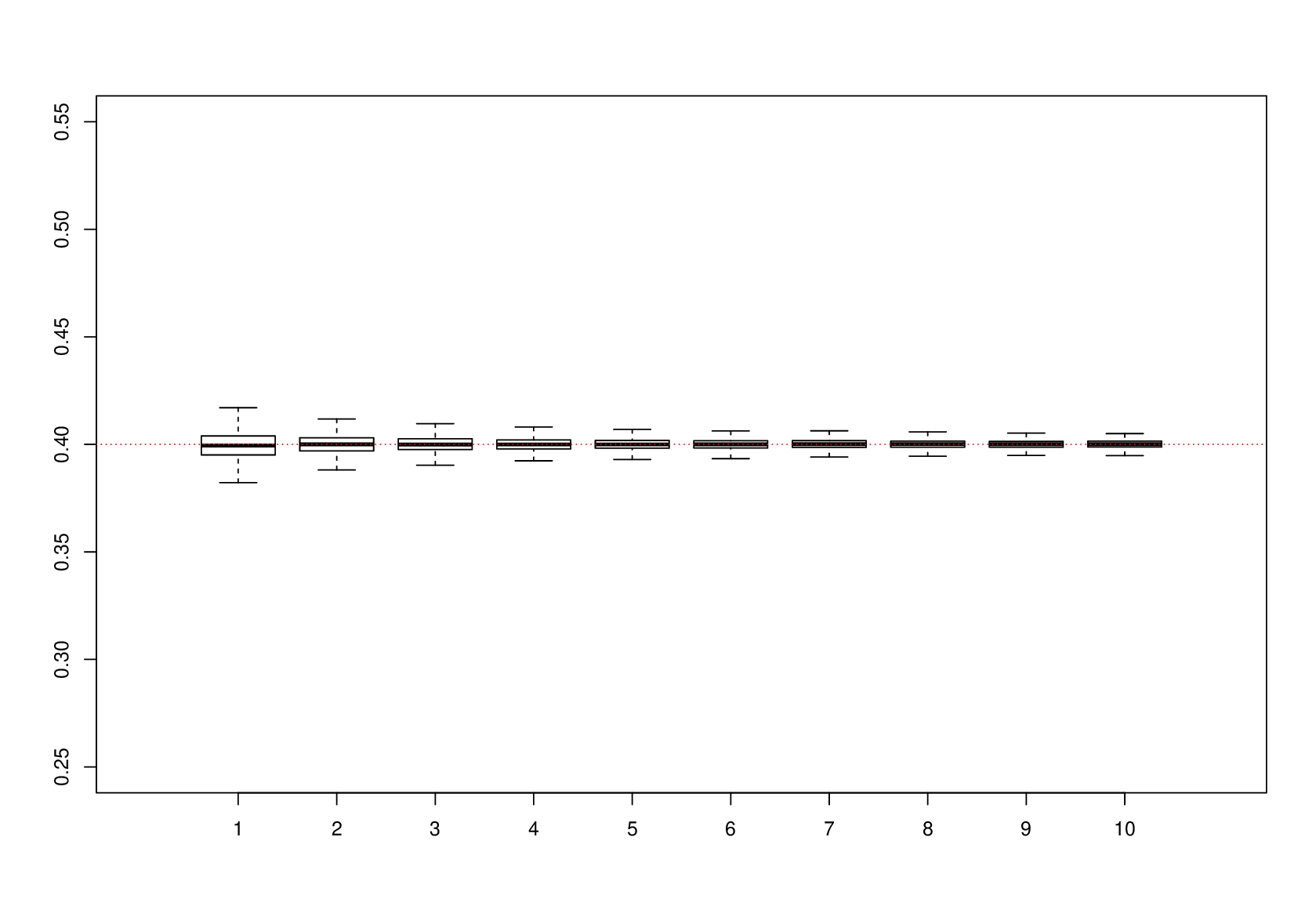} \\
  \includegraphics[angle=270,width=12.5cm]{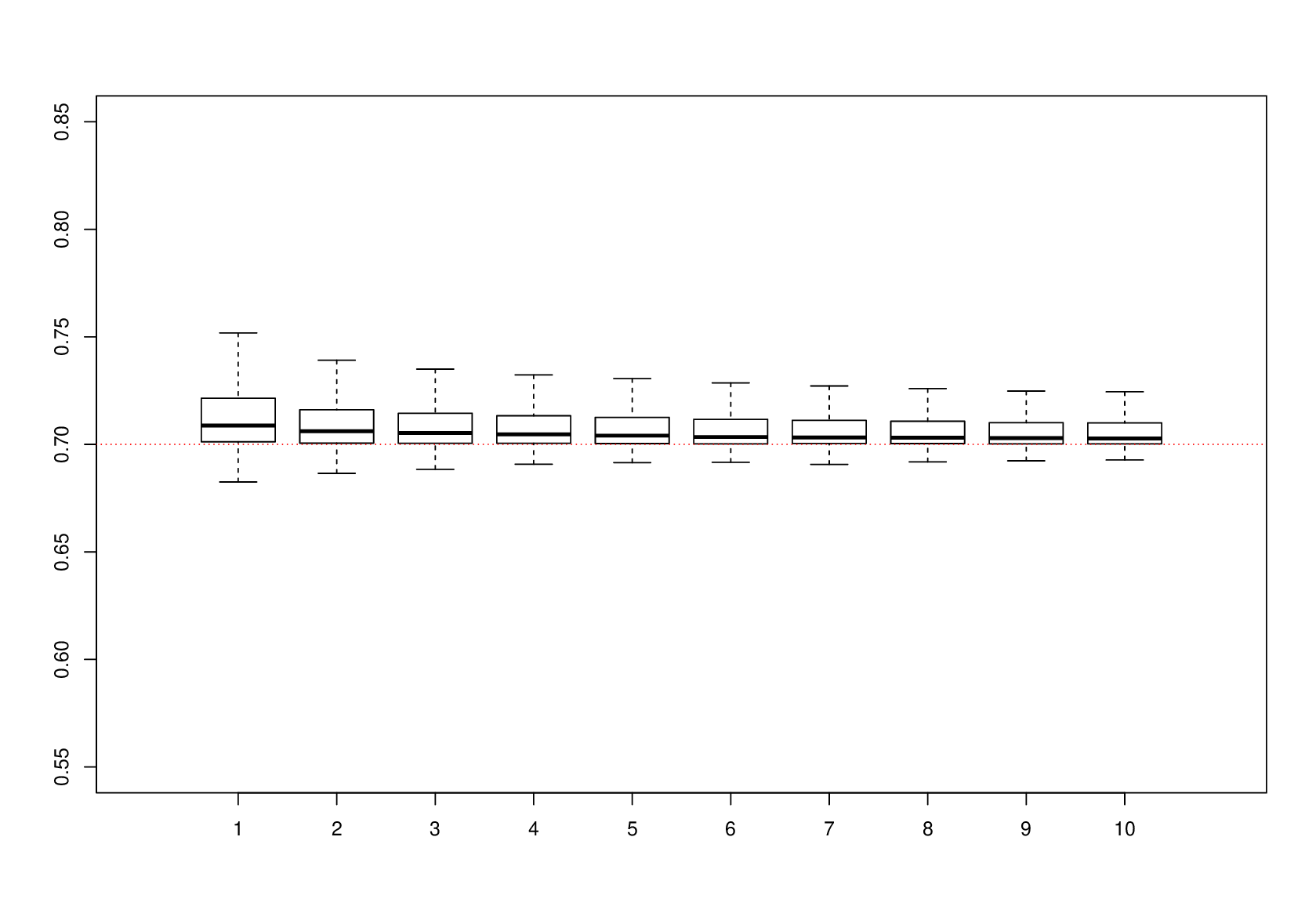} \\
  \caption{Boxplots   of   our   estimator   obtained from 1,000
    iterations in   Example~\ref{ex:2points} and for values $n$ ranging
    in $ \{10^4 \cdot k \, : \, 1\le k \le 10 \}$ ($x$-axis indicates the value $k$).
Estimation of $a_1^\star$ (top panel) and $a_2^\star$ (bottom panel). The true values
    are indicated by  horizontal lines.}
  \label{figu:boxplots_cas2}
\end{figure}

\begin{figure}[H]
  \centering
  \includegraphics[angle=270,width=12.5cm]{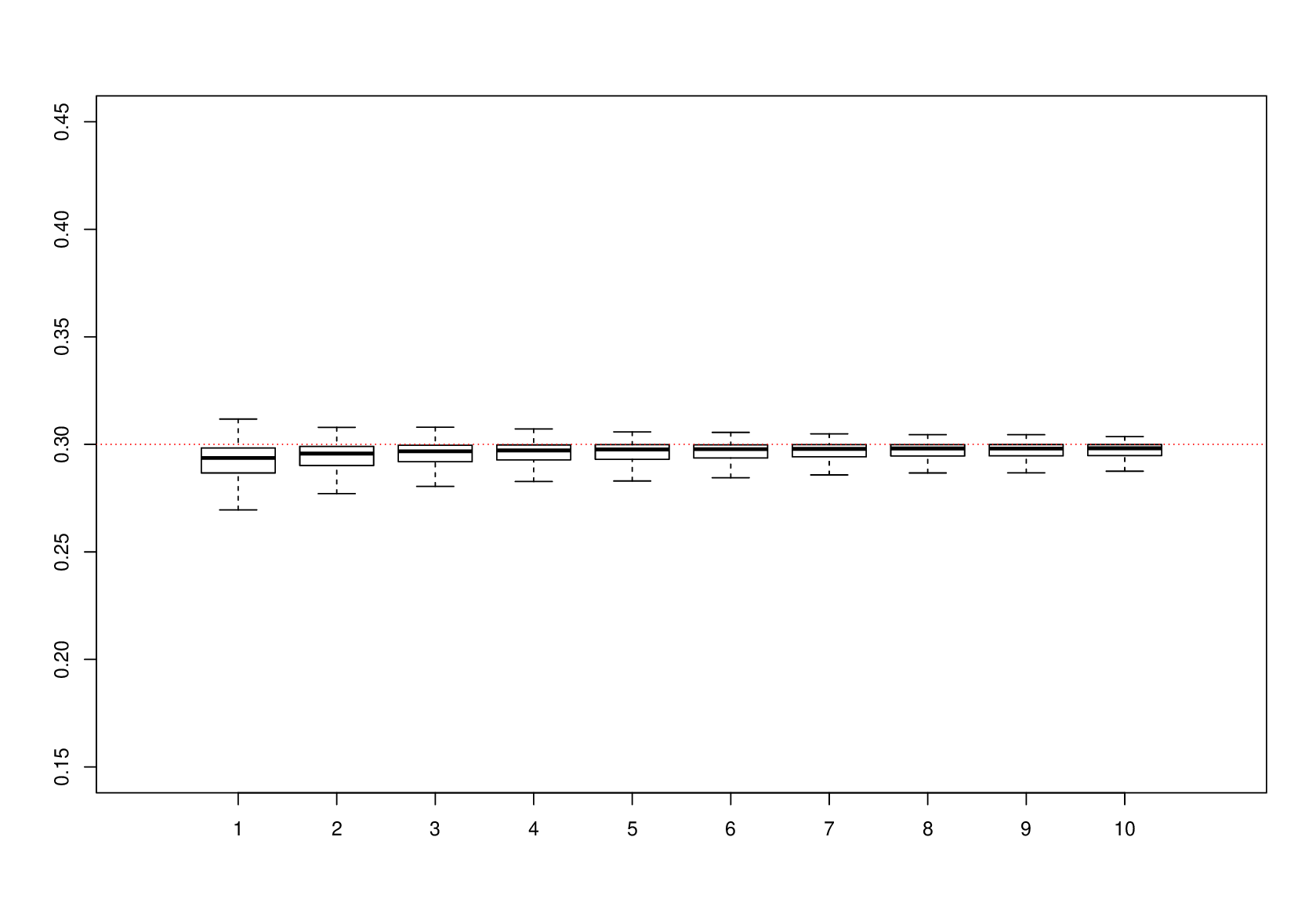} \\
  \includegraphics[angle=270,width=12.5cm]{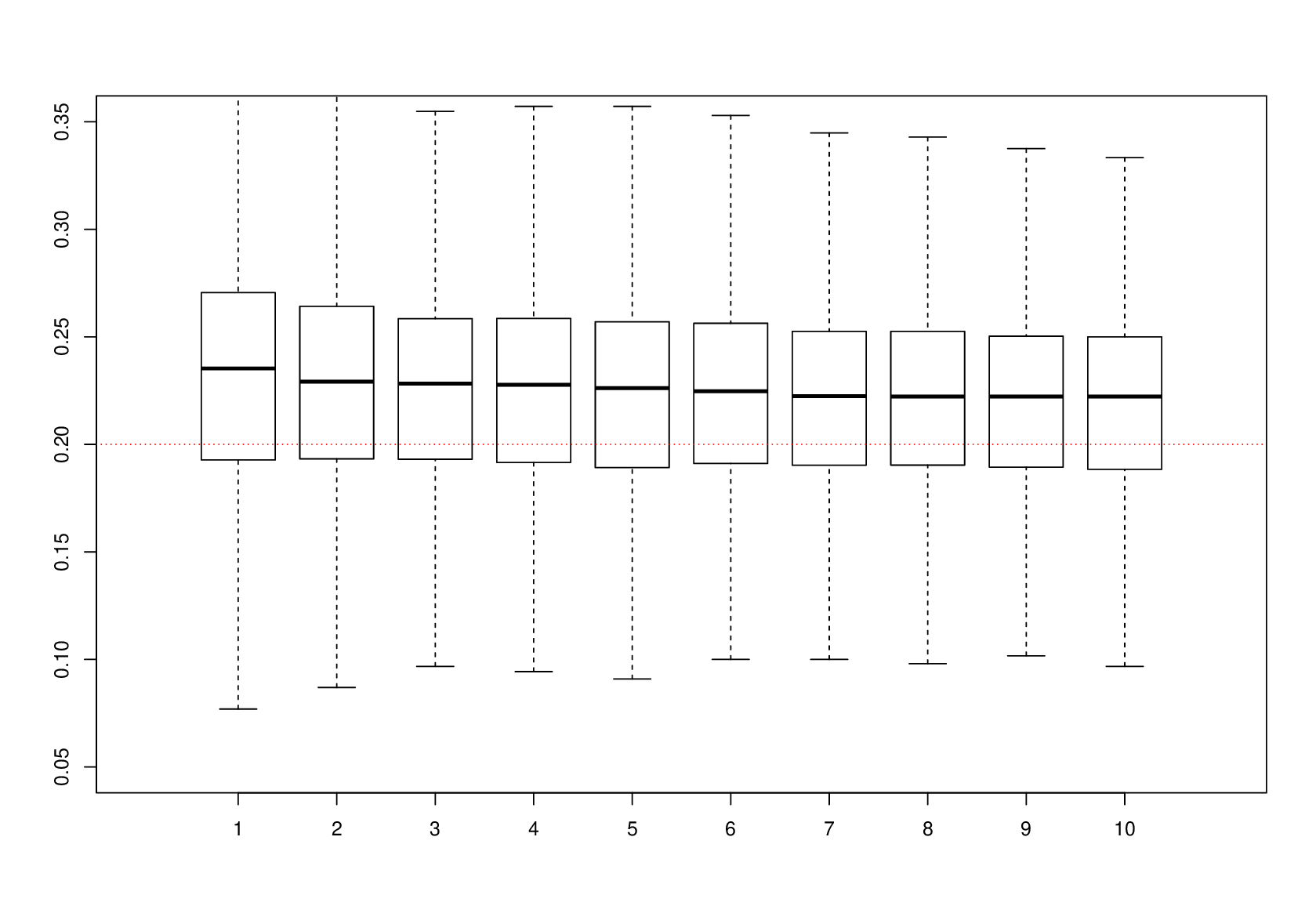} \\
  \caption{Boxplots   of   our   estimator   obtained from 1,000
    iterations in   Example~\ref{ex:lazy} and for values $n$ ranging
    in $ \{10^4 \cdot k \, : \, 1\le k \le 10 \}$ ($x$-axis indicates the value $k$).
Estimation of $a^\star$ (top panel) and $r^\star$ (bottom panel). The true values
    are indicated by  horizontal lines.}
  \label{figu:boxplots_cas3}
\end{figure}

\section*{Acknowledgments} 
The authors warmly thank Julien Chiquet for sharing many fruitful reflexions about  programming syntax and parallel computing. 

\bibliographystyle{chicago}
\bibliography{MAMAFrancis}

\end{document}

%% file: FigurePotentiel.tex
\begin{tikzpicture}[xscale=0.1,yscale=0.5]

\draw[->] (0,-8) -- (0,4);
\draw (-0.5,4) node[anchor=east] {$V(x)$}; 

\draw[->] (0,0) -- (110,0);
\draw (110,-0.5) node[anchor=north] {$x$}; 

\draw[dotted] (57,-7.625681e+00) -- (57,0);
\draw (57,0) node[anchor=south] {$b_n$}; 

\draw[dotted] (99, 2.541894e+00) -- (99,0);
\draw (99,0) node[anchor=north] {$c_n$};
 
\draw[dotted] (56,-7.625681e+00 ) -- (106,-7.625681e+00);
\draw[dotted]  (86,-7.625681e+00  + 9.536  )  -- (106,-7.625681e+00  +
9.536 );
\draw[<->,dotted] (105,-7.625681e+00 ) -- (105,-7.625681e+00 + 9.536);

\node[fill=white] at (102,-4) {$\log n + \sqrt{\log n}$};

\draw[-] (0,0.000000e+00)  -- (1,8.472979e-01) --  (2,1.694596e+00) --
(3,  8.472979e-01)  --   (4,  0)  --  (5,-8.472979e-01)  --
(6,0)   --  (7,   8.472979e-01)   --  (8,0)   --
(9,-8.472979e-01)   --  (10,0)  --   (11,-8.472979e-01)  --
(12,-1.694596e+00)  --  (13,-8.472979e-01)  --  (14,-1.694596e+00)  --
(15,-8.472979e-01)  --  (16,-1.694596e+00)  --  (17,-2.541894e+00)  --
(18,-3.389191e+00) -- (19,-4.236489e+00) -- (20,-5.083787e+00) -- 
(21,-5.931085e+00)  --  (22,-5.083787e+00)  --  (23,-4.236489e+00)  --
(24,-3.389191e+00)  --   (25,-2.541894e+00)  --  (26,-3.389191e+00)  --
(27,-2.541894e+00)  --  (28,-1.694596e+00)  --  (29,-2.541894e+00)  --
(30,-3.389191e+00)  --  (31,-2.541894e+00)  --  (32,-3.389191e+00)  --
(33,-2.541894e+00)  --  (34,-1.694596e+00)  --  (35,-2.541894e+00)  --
(36,-3.389191e+00)  --  (37,-2.541894e+00)  --  (38,-3.389191e+00)  --
(39,-2.541894e+00) -- (40,-1.694596e+00) -- (41,-2.541894e+00) -- 
(42,-1.694596e+00)  --  (43,-2.541894e+00)  --  (44,-3.389191e+00)  --
(45,-4.236489e+00)  --  (46,-5.083787e+00)  --  (47,-4.236489e+00)  --
(48,-3.389191e+00)  --  (49,-4.236489e+00)  --  (50,-3.389191e+00)  --
(51,-4.236489e+00)  --  (52,-3.389191e+00)  --  (53,-4.236489e+00)  --
(54,-5.083787e+00)  --  (55,-5.931085e+00)  --  (56,-6.778383e+00)  --
(57,-7.625681e+00)  --  (58,-6.778383e+00)  --  (59,-5.931085e+00)  --
(60,-6.778383e+00)  --  (61,-7.625681e+00)  --  (62,-6.778383e+00)  --
(63,-5.931085e+00)  --  (64,-6.778383e+00)  --  (65,-7.625681e+00)  --
(66,-6.778383e+00)  --  (67,-7.625681e+00)  --  (68,-6.778383e+00)  --
(69,-7.625681e+00)  --  (70,-6.778383e+00)  --  (71,-5.931085e+00)  --
(72,-5.083787e+00)  --  (73,-4.236489e+00)  --  (74,-5.083787e+00)  --
(75,-4.236489e+00)  --  (76,-5.083787e+00)  --  (77,-4.236489e+00)  --
(78,-3.389191e+00)  --  (79,-4.236489e+00)  --  (80,-5.083787e+00)  --
(81,-4.236489e+00)  --  (82,-5.083787e+00)  --  (83,-4.236489e+00)  --
(84,-3.389191e+00)  --  (85,-2.541894e+00)  --  (86,-3.389191e+00)  --
(87,-2.541894e+00)  --  (88,-1.694596e+00)  --  (89,-8.472979e-01)  --
(90,-1.694596e+00)  --  (91,-2.541894e+00)  --  (92,-1.694596e+00)  --
(93,-8.472979e-01)  --  (94,0)   --  (95,8.472979e-01)  --  (96,0)  --
(97,8.472979e-01)  --  (98,1.694596e+00)   --  (99,  2.541894e+00)  --
(100,3.389191e+00);

\end{tikzpicture}